\documentclass[final,leqno]{siamltex704}
\usepackage{amsmath}
\usepackage{graphicx}
\usepackage[notcite,notref]{showkeys}
\usepackage{float}
\usepackage{amsfonts,amssymb}
\usepackage{dsfont}
\usepackage{pifont}
\numberwithin{equation}{section}
\def\3bar{{|\hspace{-.02in}|\hspace{-.02in}|}}
\def\E{{\mathcal{E}}}
\def\T{{\mathcal{T}}}
\def\Q{{\mathcal{Q}}}

\def\pT{{\partial T}}

\def\LL{{\mathcal{L}}}

\def\bu{{\mathbf{u}}}
\def\bv{{\mathbf{v}}}
\def\bn{{\mathbf{n}}}
\def\be{{\mathbf{e}}}

\newtheorem{remark}{Remark}[section]
\newtheorem{algorithm}{Algorithm}[section]
\def\argmin{\operatornamewithlimits{arg\ min}}

\title {A Primal-Dual Weak Galerkin Finite Element Method for
Second Order Elliptic Equations in Non-Divergence Form}

\author{
Chunmei Wang\thanks{School of Mathematics, Georgia Institute of
Technology, Atlanta, Georgia, 30332; Taizhou College, Nanjing Normal
University, Taizhou, China 225300. The research of Chunmei Wang was
partially supported by National Science Foundation award
\#DMS-1522586.} \and Junping Wang\thanks{Division of Mathematical
Sciences, National Science Foundation, Arlington, VA 22230
(jwang@nsf.gov). The research of Junping Wang was supported by the
NSF IR/D program, while working at National Science Foundation.
However, any opinion, finding, and conclusions or recommendations
expressed in this material are those of the author and do not
necessarily reflect the views of the National Science Foundation.}}

\begin{document}

\maketitle

\begin{abstract}
This article proposes a new numerical algorithm for second order
elliptic equations in non-divergence form. The new method is based
on a discrete weak Hessian operator locally constructed by following
the weak Galerkin strategy. The numerical solution is characterized
as a minimization of a non-negative quadratic functional with
constraints that mimic the second order elliptic equation by using
the discrete weak Hessian. The resulting Euler-Lagrange equation
offers a symmetric finite element scheme involving both the primal
and a dual variable known as the Lagrange multiplier, and thus the
name of {\em primal-dual weak Galerkin finite element method}.
Optimal order error estimates are derived for the finite element
approximations in a discrete $H^2$-norm, as well as the usual $H^1$-
and $L^2$-norms. Some numerical results are presented for smooth and
non-smooth coefficients on convex and non-convex domains.
\end{abstract}



\pagestyle{myheadings}

\section{Introduction}
This paper is concerned with development of numerical methods for
second order elliptic problems in non-divergence form. For
simplicity, we consider the model problem that seeks an unknown
function $u=u(x)$ satisfying
\begin{equation}\label{1}
\begin{split}
\sum_{i,j=1}^d a_{ij}\partial^2_{ij}u&=f,\quad \text{in}\
\Omega,\\
u &=0,\quad \text{on}\ \partial\Omega,
\end{split}
\end{equation}
where $\Omega$ is an open bounded domain in $\mathbb R^d(d=2,3)$
with Lipschitz continuous boundary $\partial\Omega$,
$\LL:=\sum_{i,j=1}^d a_{ij}\partial_{ij}^2$ is the second order
partial differential operator with coefficients $a_{ij}\in
L^\infty(\Omega)$, and $f\in L^2(\Omega)$ is a given function.

Assume that the coefficient tensor $a(x)=(a_{ij}(x))_{d\times d}$ is
symmetric, uniformly bounded and positive definite. Namely, there
exist positive constants $\alpha$ and $\beta$ such that
\begin{equation}\label{matrix}
\alpha\xi^T\xi\leq \xi^T a(x)\xi\le \, \beta\xi^T\xi \qquad \forall
\xi\in \mathbb{R}^d,\ x\in\Omega.
\end{equation}
If the coefficient tensor $a(x)$ is smooth in the domain $\Omega$,
then the operator $\LL$ can be written in a divergence form:
$$
\LL u = \sum_{i,j=1}^d \partial_j (a_{ij}\partial_i u) -
\sum_{i,j=1}^d (\partial_j a_{ij}) \partial_i u
$$
so that the existing finite element methods (see
\cite{ciarlet-fem,brenner} for example) can be employed for an
accurate approximation of the problem (\ref{1}). In this paper, we
assume that the coefficient tensor $a(x)\in L^\infty(\Omega)$ is
non-smooth so that a variational formulation using integration by
parts is not possible.

Problems in the form of (\ref{1}) arise in many applications from
applied areas such as probability and stochastic processes
\cite{Fleming}. They also appear in the study of fully nonlinear
partial differential equations in conjunction with linearization
techniques such as the Newton's iterative method \cite{brenner-0,
neilan}. In many such applications, the coefficient tensor $a(x)$ is
hardly smooth nor even continuous. For example, the coefficient
$a(x)$ is merely essentially bounded in the application to
Hamilton-Jacobi-Bellman equations \cite{Fleming}. For fully
nonlinear PDEs discretized by discontinuous finite elements, their
linearization involves at most piecewise smooth coefficients.
Therefore, it is important and crucial to develop efficient
numerical methods for Problem (\ref{1}) with rough coefficient
tensor.

Several numerical methods were recently designed and studied for
PDEs in non-divergence form by using finite element approaches based
on ad-hoc variational forms. In \cite{lakkis}, a Galerkin type
method was introduced by using conforming finite elements in the
computing of a {\em finite element Hessian}. This finite element
scheme was further modified and analyzed in \cite{neilan}. In
\cite{feng}, a nonstandard primal finite element method, which uses
finite-dimensional subspaces consisting globally continuous
piecewise polynomial functions, was proposed and analyzed. The key
in \cite{feng} is the use of an interior penalty term, which
penalizes the jump of the flux across the interior element
edges/faces, to augment a nonsymmetric piecewise defined and
PDE-induced bilinear form. In \cite{smears}, an $hp$-version
discontinuous Galerkin finite element method of least-squares type
was designed and analyzed for a class of such problems that satisfy
the Cord\`es condition. In particular, the authors showed that the
method exhibits a convergence rate that is optimal with respect to
the mesh size $h$ and suboptimal with respect to the polynomial
degree $p$ by half an order.

The goal of this paper is to develop a new finite element method for
the model problem (\ref{1}) by using the weak Galerkin strategy
recently introduced in \cite{wy2013, wy2707, mwy3655, wy3655} for
partial differential equations. One of the two basic principles for
weak Galerkin is the use of locally constructed differential
operators, called {\em discrete weak differential operators}, in the
space of discontinuous functions including necessary boundary
information. The discrete weak differential operators form the
critical building block in discretization of the underlying PDEs.
For the model problem (\ref{1}), Hessian is the primary differential
operator which shall be locally reconstructed by using the weak
Galerkin approach. The resulting discrete weak Hessian, denoted by
$\{\partial^2_{ij,d} v\}_{d\times d}$ to be detailed in Sections
\ref{Section:Hessian} and \ref{Section:PD-WGFEM}, is then employed
to approximate (\ref{1}) as follows
\begin{equation}\label{EQ:10-12-2015:01}
\sum_{i,j=1}^d (a_{ij} \partial^2_{ij,d} u_h, w) = (f, w), \qquad
\forall w\in W_{h,k},
\end{equation}
where $W_{h,k}$ is a test space and $u_h$ is sought from a trial
space $V_{h,k}$. The discrete problem (\ref{EQ:10-12-2015:01}),
however, is not well-posed unless an {\em inf-sup} condition of
Babu\u{s}ka \cite{babuska} and Brezzi \cite{b1974} is satisfied. To
overcome this difficulty, this paper proposes a constraint
optimization algorithm which seeks $u_h\in V_{h,k}$ as a
minimization of a prescribed non-negative quadratic functional
$J(v)=\frac12 s(v,v)$ with constraint given by the equation
(\ref{EQ:10-12-2015:01}). The functional $J(v)$ measures the
``continuity" of $v\in V_{h,k}$ in the sense that $v\in V_{h,k}$ is
a classical conforming element if and only if $s(v,v)=0$. The weak
continuity of the finite element approximation $u_h$ as
characterized by the functional $J(v)$ forms the second basic
principle of weak Galerkin. The resulting Euler-Lagrange equation
for the constraint optimization problem gives rise to a symmetric
numerical algorithm involving not only the primal variable $u_h$ but
also a dual variable $\lambda_h$ known as the Lagrange multiplier.
This numerical scheme, called {\em primal-dual weak Gelerkin finite
element method}, is the main contribution of the present paper.

Our theory for the primal-dual weak Gelerkin finite element method
is based on the assumption that the solution of (\ref{1}) is
$H^2$-regular, and that the coefficient tensor $a(x)$ is piecewise
continuous and satisfies the uniform ellipticity condition
(\ref{matrix}). Under those assumptions, an optimal order error
estimate is derived in a discrete $H^2$-norm for the primal variable
and in the $L^2$-norm for the dual variable. We shall also establish
a convergence theory for the primal variable in the $H^1$- and
$L^2$-norms under some smoothness assumptions for the coefficient
tensor. Numerical experiments are presented to illustrate the
accuracy and to confirm the theory developed for the primal-dual
weak Galerkin finite element method.

The paper is organized as follows. In Section
\ref{Section:Preliminaries}, we present some preliminary results on
strong solutions for the model problem (\ref{1}). Section
\ref{Section:Hessian} is devoted to a discussion of weak Hessian and
its discretizations. In Section \ref{Section:PD-WGFEM}, we describe
the primal-dual weak Galerkin finite element method for the model
problem (\ref{1}). Section \ref{Section:ExistenceUniquess} is
devoted to a stability analysis for the new finite element method.
In Section \ref{Section:ErrorEstimates}, we derive an optimal order
error estimate for the numerical method in a discrete $H^2$-norm for
piecewise continuous coefficient tensors. Section
\ref{Section:H1L2Error} continues the error analysis by establishing
some error estimates in the usual $H^1$- and $L^2$-norms for the
primal variable under some smoothness assumptions on the coefficient
tensor. Finally in Section \ref{Section:NE}, we conduct some
numerical experiments for the model problem (\ref{1}) with smooth
and non-smooth coefficients $a(x)$ on convex and non-convex domains.

\section{Preliminaries}\label{Section:Preliminaries}

Let $D\subset \mathbb{R}^d$ be an open bounded domain with Lipschitz
continuous boundary. We use the standard definition for the Sobolev
space $H^s(D)$ and the associated inner product
$(\cdot,\cdot)_{s,D}$, norm $\|\cdot\|_{s,D}$, and seminorm
$|\cdot|_{s,D}$ for any $s\ge 0$ \cite{ciarlet-fem, brenner}. We
also use $\langle\cdot,\cdot\rangle_{\partial D}$ to denote the
usual inner products in $L^2(\partial D)$. For simplicity, we shall
drop the subscript $D$ in the norm and inner product notation when
$D=\Omega$. In addition, $\|\cdot\|_{0,D}$ and
$\|\cdot\|_{0,\partial D}$ are simplified as $\|\cdot\|_{D}$ and
$\|\cdot\|_{\partial D}$, respectively.

The classical Schauder's theory \cite{Gilbarg-Trudinger} states that
if the coefficient matrix $a=a(x)$ is of $C^{0,\alpha}(\Omega)$ and
$\partial \Omega \in C^{2,\alpha}$, then there exists a unique
solution $u\in C^{2,\alpha}(\Omega)$ satisfying the model problem
(\ref{1}). The Calder\'on-Zygmund theory states that if $a=a(x)$ is
of $C^{0}(\bar{\Omega})$ and $\partial \Omega \in C^{1,1}$, then
there exists a unique solution $u\in W^{2,p}(\Omega)$ satisfying
(\ref{1}), see Theorem 9.15 in \cite{Gilbarg-Trudinger} for details.
Furthermore, one has the following a priori estimate
\begin{equation}\label{EQ:W2pregularity}
\|u\|_{2,p} \leq C \|f\|_{0,p}.
\end{equation}
Here $p\in (1,\infty)$ is any given real number.

The solution uniqueness may break down when $d\ge 3$ for
coefficients $a(x)$ that are not continuous. One such example is
given by
\begin{equation}\label{nonuniqueness-example}
a(x) = I_{d\times d} +\frac{(d+\lambda-2) x x^T}{(1-\lambda) |x|^2}.
\end{equation}
With $\Omega=B_1(0), d>2(2-\lambda)$, it can be verified that
$u=|x|^\lambda \in H^2(\Omega)\cap H_0^1(\Omega)$ satisfies the
partial differential equation in (\ref{1}) with $f=0$. For this
reason, in the case $a(x)$ is discontinuous, we assume the following
Cord\`es condition is satisfied: There exists an $\varepsilon\in
(0,1]$ such that
\begin{equation}\label{cordes}
\frac{\sum_{i,j=1}^d a_{ij}^2}{(\sum_{i,j=1}^d a_{ii})^2} \leq
\frac{1}{d-1+\varepsilon}\qquad \mbox{in}\ \Omega.
\end{equation}

\begin{theorem}\label{THM:H2regularity} \cite{smears} Let $\Omega\subset \mathbb R^d$ be a
bounded convex domain, and let the differential operator defined in
(\ref{1}) satisfy $a\in [L^\infty(\Omega)]_{d\times d}$, the
ellipticity condition (\ref{matrix}), and the Cord\`es condition
(\ref{cordes}). Then, for any given $f\in L^2(\Omega)$, there exists
a unique $u\in H^2(\Omega)\cap H_0^1(\Omega)$ that is a strong
solution of (\ref{1}), and this strong solution satisfies
\begin{equation}\label{EQ:H2regularity}
\|u\|_{2} \leq C \|f\|_0,
\end{equation}
where $C$ is a constant depending only on $d$, the diameter of
$\Omega$, $\alpha$, $\beta$, and $\varepsilon$.
\end{theorem}

For problems in two dimensions, the uniform ellipticity assumption
(\ref{matrix}) implies the validity of the Cord\`es condition
(\ref{cordes}), see \cite{smears} and the references cited therein.
In fact, let $\lambda_{min}(x)$ and $\lambda_{max}(x)$ be the
smallest and the largest eigenvalues of $a(x)$. It is easy to see
that $a_{ii}\leq \lambda_{max}$ for $i=1,2$, and
$a_{11}a_{22}-a_{12}^2 = \lambda_{min}\lambda_{max}$. It follows
that
\begin{equation*}
\begin{split}
(\sum_{i,j=1}^2 a_{ii})^2 \leq & 4 \lambda_{max}^2 =
4\frac{\lambda_{max}}{\lambda_{min}}
\lambda_{min}\lambda_{max}\\
 = & 4\frac{\lambda_{max}}{\lambda_{min}}
\left(a_{11}a_{22}-a_{12}^2 \right)\\
= & 4\kappa(a(x)) \left(a_{11}a_{22}-a_{12}^2 \right),
\end{split}
\end{equation*}
where $\kappa(a(x))$ is the condition number of the matrix $a(x)$.
Thus, we have
\begin{equation}\label{cordes-new}
\begin{split}
\frac{\sum_{i,j=1}^2 a_{ij}^2}{(\sum_{i,j=1}^2 a_{ii})^2} & =
\frac{(\sum_{i,j=1}^2 a_{ii})^2 - 2 (a_{11}a_{22}-a_{12}^2
)}{(\sum_{i,j=1}^2 a_{ii})^2} \\
& \leq 1- \frac{1}{2\kappa(a(x))} = \frac{1}{1+
\frac{1}{2\kappa(a)-1}}
\end{split}
\end{equation}
for all $x\in \Omega$. The last inequality is exactly the Cord\`es
condition (\ref{cordes}) with $\varepsilon =\frac{1}{2\kappa(a)-1}$.
Note that the uniform ellipticity (\ref{matrix}) implies
$\kappa(a)\leq \beta/\alpha$. Hence, the Cord\`es condition
(\ref{cordes}) is satisfied with $\varepsilon
=\frac{\alpha}{2\beta-\alpha}$ under the condition of (\ref{matrix})
for two dimensional problems.

\medskip
Throughout this paper, we assume that the problem (\ref{1}) has a
unique strong solution in $H^2(\Omega)\cap H_0^1(\Omega)$ with the
following a priori estimate
\begin{equation}\label{Assumption:H2regularity}
\|u\|_{2} \leq C \|f\|_0,
\end{equation}
where $C$ is a generic constant which represents different values at
different appearances.

Let $X=H^2(\Omega)\cap H_0^1(\Omega)$ and $Y=L^2(\Omega)$. Introduce
the following bilinear form in $X\times Y$:
\begin{equation}\label{b-form}
b(v, \sigma) := (\LL v, \sigma), \qquad v\in X, \ \sigma\in Y.
\end{equation}
Then, the strong solution of the problem (\ref{1}) satisfies the
following variational equation: Find $u\in X$ such that
\begin{equation}\label{Abstract-Formulation}
b(u,w) = (f, w)\qquad \forall w\in Y.
\end{equation}
It follows from the regularity assumption
\ref{Assumption:H2regularity} that the bilinear form
$b(\cdot,\cdot)$ satisfies the following {\em inf-sup} condition
$$
\sup_{v\in X, v\neq 0} \frac{b(v,\sigma)}{\|v\|_X} \ge \Lambda
\|\sigma\|_Y
$$
for all $\sigma\in Y$, where $\Lambda$ is a generic constant related
to the constant $C$ in the $H^2$ regularity estimate
(\ref{Assumption:H2regularity}). Here $\|\cdot\|_X$ stands for the
$H^2(\Omega)$-norm, and $\|\cdot\|_Y$ is the standard
$L^2(\Omega)$-norm.

\begin{remark}
If the problem (\ref{1}) has the $W^{2,p}$-regularity
(\ref{EQ:W2pregularity}) instead of (\ref{Assumption:H2regularity}),
then the variational equation (\ref{Abstract-Formulation}) still
holds true with $X=W^{2,p}(\Omega)\cap W^{1,p}_0(\Omega)$ and
$Y=L^q(\Omega)$, where $q$ is the conjugate of $p\in (1,\infty)$ so
that $p^{-1}+q^{-1}=1$.
\end{remark}

\section{Weak Hessian and Discrete Weak
Hessian}\label{Section:Hessian} For classical functions, the Hessian
is a square matrix of second order partial derivatives if they all
exist. Note that Hessian is the primary differential operator in the
composition of the second order elliptic problem (\ref{1}) in the
non-divergence form. It is therefore necessary to develop numerical
techniques targeted at the Hessian operator. The objective of this
section is to review the discrete weak Hessian operator introduced
in \cite{ww}.

Let $K$ be a polygonal or polyhedral domain with boundary $\partial
K$. By a weak function on $K$ we mean a triplet
$v=\{v_0,v_b,\bv_g\}$ such that $v_0\in L^2(K)$, $v_b\in
L^{2}(\partial K)$ and $\bv_g\in [L^{2}(\partial K)]^d$. The first
and second components, namely $v_0$ and $v_b$, represent the value
of $v$ in the interior and on the boundary of $K$. The third one,
$\bv_g=(v_{g1},\ldots,v_{gd})\in \mathbb{R}^d$, intends to represent
the gradient vector $\nabla v$ on the boundary of $K$. Note that
$v_b$ and $\bv_g$ may or may not be related to the trace of $v_0$
and $\nabla v_0$ on $\partial K$. In the case of traces are used (if
they exist), the weak function $v$ is uniquely determined by its
first component $v_0$, and it becomes to be a classical function. It
is also possible to take $v_b$ as the trace of $v_0$ and leave
$\bv_g$ completely free or vice versa. Denote by $W(K)$ the space of
all weak functions on $K$
\begin{equation}\label{2.1}
W(K)=\{v=\{v_0,v_b,\bv_g\}: v_0\in L^2(K), v_b\in L^{2}(\partial K),
\bv_g\in [L^{2}(\partial K)]^d\}.
\end{equation}

For any $v\in W(K)$, the generalized weak second order partial
derivative is defined as a bounded linear functional
$\partial^2_{ij,w} v$ on the Sobolev space $H^2(K)$ so that its
action on each $\varphi\in H^2(K)$ is given by
 \begin{equation}\label{2.3}
 \langle \partial^2_{ij,w}v,\varphi\rangle_K:=(v_0,\partial^2_{ji}\varphi)_K-
 \langle v_b n_i,\partial_j\varphi\rangle_{\partial K}+
 \langle v_{gi},\varphi n_j\rangle_{\partial K}.
 \end{equation}
Here, $\bn=(n_1,\cdots,n_d)$ is the unit outward normal direction on
$\partial K$. The weak Hessian of $v\in W(K)$ is defined as
$\nabla^2_{w,K} v  = \left\{\partial_{ij,w}^2 v \right\}_{d\times
d}.$

Let $S_r(K)$ be a finite dimensional linear space consisting of
polynomials on $K$. A discrete analogy of $\partial^2_{ij,w}$,
denoted by $\partial^2_{ij,w,r,K}$, is defined as the unique
polynomial $\partial^2_{ij,w,r,K} v\in S_r(K)$ such that
  \begin{equation}\label{2.4}
 (\partial^2_{ij,w,r,K}v,\varphi)_K=(v_0,\partial^2
 _{ji}\varphi)_K-\langle v_b n_i,\partial_j\varphi\rangle_{\partial K}
 +\langle v_{gi},\varphi n_j\rangle_{\partial K},\quad \forall \varphi \in
 S_r(K).
 \end{equation}
Analogously, for any $v\in W(K)$, its discrete weak Hessian is given
by
$$
\nabla^2_{w,r,K} v = \left\{\partial_{ij,w,r,K}^2 v
\right\}_{d\times d}.
$$

If $v\in W(K)$ has a smooth component $v_0\in H^2(K)$, then the
usual integration by parts can be applied to the first term on the
right-hand side of (\ref{2.4}), yielding
\begin{equation}\label{2.4new}
 (\partial^2_{ij,w,r,K}v,\varphi)_K=(\partial^2
 _{ij}v_0,\varphi)_K-\langle (v_b-v_0) n_i,\partial_j\varphi\rangle_{\partial K}
 +\langle v_{gi}-\partial_i v_0,\varphi n_j\rangle_{\partial K},
 \end{equation}
 for all $\varphi \in S_r(K)$.

\section{Primal-Dual Weak Galerkin}\label{Section:PD-WGFEM}
Let ${\cal T}_h$ be a finite element partition of the domain
$\Omega$ into polygons in 2D or polyhedra in 3D. Denote by
${\mathcal E}_h$ the set of all edges or flat faces in ${\cal T}_h$
and ${\mathcal E}_h^0={\mathcal E}_h \setminus \partial\Omega$ the
set of all interior edges or flat faces. Assume that ${\cal T}_h$
satisfies the shape regularity conditions described as in
\cite{wy3655}. Denote by $h_T$ the diameter of $T\in {\cal T}_h$ and
$h=\max_{T\in {\cal T}_h}h_T$ the meshsize of the partition ${\cal
T}_h$. For any integer $m\ge 0$, denote by $P_m(T)$ the set of all
polynomials of total degree $m$ or less.

For any given integer $k\geq 2$, let $W_k(T)\subset W(T)$ be a
subspace consisting of (piecewise) polynomials in the following form
\begin{equation}\label{EQ:local-weak-fem-space}
W_k(T):=\{v=\{v_0,v_b,\bv_g\}\in P_k(T)\times P_k(e)\times
[P_{k-1}(e)]^d,\ e\in \partial T\cap\E_h\}.
\end{equation}
By patching $W_k(T)$ over all $T\in {\cal T}_h$ through a common
value on the interface $\E_h^0$ for $v_b$ and $\bv_g$, we arrive at
the following weak finite element space
$$
W_{h,k}:=\big\{\{v_0,v_b, \textbf{v}_g\}:\ \{v_0,v_b, \bv_g\}|_T\in
W_k(T), \ T\in {\cal T}_h\big\}.
$$
Denote by $W_{h,k}^0$ the subspace of $W_{h,k}$ with vanishing
boundary value for $v_b$ on $\partial\Omega$:
\begin{equation}\label{EQ:global-weak-fem-space}
W_{h,k}^0=\{\{v_0,v_b, \bv_g\}\in W_{h,k},\ v_b|_e=0, e\subset
\partial\Omega\}.
\end{equation}

Next, let $S_k(T)$ be a linear space of polynomials satisfying
\begin{equation}\label{EQ:LocalS}
P_{k-2}(T) \subseteq S_k(T) \subseteq P_{k-1}(T).
\end{equation}
Correspondingly, we have the following finite element space
\begin{equation}\label{EQ:GlobalS}
S_{h,k}=\Big\{\sigma:\ \sigma|_T\in S_k(T),\ T\in {\cal T}_h\Big\}.
\end{equation}

For simplicity of notation, we denote by $\partial^2_{ij, d}$ the
discrete weak second order partial differential operator defined by
(\ref{2.4}) with $S_r(T)=S_k(T)$ on each element $T$; i.e.,
$$
(\partial^2_{ij, d} v)|_T=\partial^2_{ij,w,r,T}(v|_T), \qquad v\in
W_{h,k}.
$$
On each element $T$, we introduce
\begin{eqnarray}\label{EQ:local-b-form}
b_T(v,\sigma)&=&\sum_{i,j=1}^d (a_{ij}\partial_{ij,d}^2
v,\sigma)_T,\\
s_T(u,v)&=&h_T^{-3}\langle u_0-u_b, v_0-v_b\rangle_\pT +
h_T^{-1}\langle \nabla u_0 -\textbf{u}_g, \nabla v_0
-\textbf{v}_g\rangle_\pT,\label{EQ:local-s-form}
\end{eqnarray}
for $u, v \in W_k(T)$ and $\sigma\in S_k(T)$. Summing up over
$T\in\T_h$ gives the following two bilinear forms
\begin{eqnarray}\label{EQ:global-b-form}
b_h(v, \sigma)&=&\sum_{T\in {\cal T}_h}b_T(v, \sigma),\quad v\in
W_{h,k}, \ \sigma\in S_{h,k},\\
s_h(u,v)&=&\sum_{T\in {\cal T}_h}s_T(u,v), \quad u,v\in
W_{h,k}.\label{EQ:global-s-form}
\end{eqnarray}

Using the bilinear forms defined in (\ref{EQ:global-b-form}) and
(\ref{EQ:global-s-form}), the second order elliptic problem
(\ref{1}) can be discretized as a constrained optimization problem
read as follows: Find $u_h\in W_{h,k}^0$ such that
\begin{equation}\label{WG:ConstrainedOpt}
u_h=\argmin_{v\in W_{h,k}^0, \ b_h(v, \sigma)=(f,\sigma),\ \forall
\sigma\in S_{h,k}} \left(\frac12 s_h(v, v)\right).
\end{equation}
The Euler-Lagrange equation for the constrained minimization problem
(\ref{WG:ConstrainedOpt}) gives rise to the following numerical
scheme.

\begin{algorithm}\emph{(Primal-Dual Weak Galerkin FEM)}
\label{ALG:primal-dual-wg-fem} For a numerical approximation of the
second order elliptic problem (\ref{1}) in the non-divergence form,
find $(u_h;\lambda_h)\in W_{h,k}^0 \times S_{h,k}$ satisfying
 \begin{eqnarray}\label{2}
 s_h(u_h,v)+b_h(v, \lambda_h)&=&0,\qquad\qquad \forall v\in W_{h,k}^0,\\
b_h(u_h, \sigma)&=&(f,\sigma),\qquad \forall\sigma\in
S_{h,k}.\label{32}
\end{eqnarray}
\end{algorithm}

From (\ref{EQ:LocalS}), the finite element space $S_k(T)$ for the
Lagrange multiplier can be chosen as any linear space between
$P_{k-2}(T)$ and $P_{k-1}(T)$. The choice of $S_k(T)=P_{k-2}(T)$ has
the least degrees of freedom, but the resulting numerical solution
may not be as accurate as the case of $S_k(T)=P_{k-1}(T)$. Some
numerical results will be presented in Section \ref{Section:NE} for
a comparison on the approximation accuracies and their order of
convergence.

\section{Stability and Solvability}\label{Section:ExistenceUniquess}
In this section, we first derive an {\it inf-sup} condition for the
bilinear form $b_h(\cdot,\cdot)$, and then show the existence and
uniqueness for the solution of the Algorithm
\ref{ALG:primal-dual-wg-fem} defined by the equations
(\ref{2})-(\ref{32}).

For each element $T$, denote by $Q_0$ the $L^2$ projection onto
$P_k(T)$, $k\geq 2$. For each edge or face $e\subset\partial T$,
denote by $Q_b$ and $\textbf{Q}_g=(Q_{g1}, Q_{g2},\ldots, Q_{gd})$
the $L^2$ projections onto $P_{k}(e)$ and $[P_{k-1}(e)]^d$,
respectively. For any $w\in H^2(\Omega)$, denote by $Q_h w$ the
$L^2$ projection onto the weak finite element space $W_{h,k}$ such
that on each element $T$,
$$
Q_hw=\{Q_0w,Q_bw,\textbf{Q}_g(\nabla w)\}.
$$
Next, denote by ${\cal Q}_h$ the $L^2$ projection onto the space
$S_{h,k}$, which is clearly a composition of local $L^2$ projections
into $S_k(T)$.

\begin{lemma}\label{Lemma5.1} \cite{ww} The projection operators $Q_h$ and
${\cal Q}_h$ satisfy the following commutative property:
\begin{equation}\label{EQ:CommutativeP}
\partial^2_{ij,d}(Q_h w)={\cal Q}_h(\partial^2_{ij} w),\qquad
i,j=1,\ldots,d,
\end{equation}
for all $w\in H^2(T)$.
\end{lemma}
\begin{proof} For any $\varphi\in S_k(T)$ and $w\in H^2(T)$, from (\ref{2.4}) and the usual
integration by parts we have
\begin{equation*}
\begin{split}
(\partial^2_{ij,d}(Q_h w),\varphi)_T&=(Q_0
w,\partial^2_{ji}\varphi)_T -\langle Q_b w,\partial_j \varphi
n_i\rangle_{\partial T}+
\langle Q_{gi}(\partial_i w),\varphi n_j\rangle_{\partial T}\\
&=(w,\partial^2_{ji}\varphi)_T-\langle w,\partial_j
 \varphi n_i\rangle_{\partial T}+
 \langle \partial_i w,\varphi n_j\rangle_{\partial T}\\
&=(\partial^2_{ij}w,\varphi)_T\\
&=({\cal Q}_h\partial^2_{ij}w,\varphi)_T.
\end{split}
\end{equation*}
It follows that (\ref{EQ:CommutativeP}) holds true. This completes
the proof of the lemma.
\end{proof}

In the weak finite element space $W_{h,k}$, let us introduce the
following semi-norm
\begin{equation}\label{EQ:triple-bar}
\| v\|^2_{2,h}=\sum_{T\in {\cal T}_h}  \|\sum_{i,j=1}^d {\cal Q}_h
(a_{ij}\partial_{ij}^2 v_0)\|_{T}^2 + s_h(v,v)
\end{equation}
The following Lemma shows that $\|\cdot\|_{2,h}$ is indeed a norm in
the subspace $W_{h,k}^0$ when the meshsize $h$ is sufficiently
small.

\begin{lemma}\label{Lemma:Sept:01} Assume that the coefficient functions $a_{ij}$ are
uniformly piecewise continuous in $\Omega$ with respect to the
finite element partition $\T_h$. There exists a fixed $h_0>0$ such
that if $v=\{v_0,v_b,\bv_g\}\in W_{h,k}^0$ satisfies
$\|v\|_{2,h}=0$, then one must have $v\equiv 0$ when $h\le h_0$.
\end{lemma}

\begin{proof}
Assume that $v=\{v_0,v_b,\bv_g\}\in W_{h,k}^0$ satisfies
$\|v\|_{2,h}=0$. It follows from (\ref{EQ:triple-bar}) and
(\ref{EQ:global-s-form}) that
\begin{equation}\label{EQ:August-30:100}
\sum_{i,j=1}^d {\cal Q}_h (a_{ij}\partial_{ij}^2 v_0) = 0,\
v_0|_\pT=v_b,\ \nabla v_0|_\pT =\textbf{v}_g
\end{equation}
for all $T\in\T_h$. Thus, $v_0 \in C^1_0(\Omega)$ and satisfies
\begin{equation}\label{EQ:August-30:200}
 \sum_{i,j=1}^d {\cal Q}_h (a_{ij}\partial_{ij}^2 v_0) = 0.
\end{equation}
Hence,
\begin{equation}\label{f-f-new}
\begin{split}
  \sum_{i,j=1}^d a_{ij}\partial_{ij}^2 v_0=& \sum_{i,j=1}^d(I-{\cal Q}_h)
\left(a_{ij}\partial_{ij}^2 v_0\right)\\
 =&\sum_{i,j=1}^d(I-{\cal Q}_h)
\left((a_{ij}-\bar{a}_{ij})\partial_{ij}^2 v_0\right)=:F,
\end{split}
\end{equation}
where $\bar{a}_{ij}$ is the average of $a_{ij}$ on $T\in\T_h$. Using
the $H^{2}$-regularity assumption (\ref{Assumption:H2regularity}),
there exists a constant $C$ such that
\begin{equation}\label{EQ:August-30:208}
\|v_0\|_{2}\leq C\|F\|_{0}.
\end{equation}
Note that $a_{ij}$ is uniformly piecewise continuous in $\Omega$
with respect to $\T_h$. Thus, for any $\varepsilon>0$, there exists
a $h_0>0$ such that $\|a_{ij} - \bar{a}_{ij}\|_{L^\infty} \leq
\varepsilon$. Using the stability of the $L^2$ projection ${\cal
Q}_h$, we arrive at
$$
\|F\|_{0} \leq C \varepsilon \|v_0\|_{2}.
$$
Substituting the above into (\ref{EQ:August-30:208}) yields
\begin{equation}\label{EQ:August-30:209}
\|v_0\|_{2}\leq C\varepsilon \|v_0\|_{2}.
\end{equation}
This implies that $v_0=0$ if $\varepsilon$ is so small that
satisfies $C\varepsilon<1$, which can be easily achieved by
adjusting the parameter $h_0$.
\end{proof}

For convenience, in the weak finite element space $W_{h,k}$, we
introduce another semi-norm
\begin{equation}\label{EQ:triple-bar-new}
\3bar v\3bar ^2_{2}=\sum_{T\in {\cal T}_h}  \|\sum_{i,j=1}^d {\cal
Q}_h (a_{ij}\partial_{ij,d}^2 v)\|_{T}^2 + s_h(v,v).
\end{equation}
Observe that the only difference between $\|v\|_{2,h}$ and $\3bar
v\3bar_{2}$ lies in the first term of (\ref{EQ:triple-bar}) and
(\ref{EQ:triple-bar-new}) where the strong second order partial
derivatives are replaced by the discrete weak second order partial
derivatives. The following Lemma shows that they are indeed
equivalent.

\begin{lemma}\label{Lemma:Sept:01-new} Assume that the coefficient functions $a_{ij}$ are
uniformly piecewise continuous in $\Omega$ with respect to the
finite element partition $\T_h$. There exist $\alpha_1>0$ and
$\alpha_2>0$ such that
\begin{equation}\label{EQ:Sept:2015:001}
\alpha_1 \|v\|_{2,h} \leq \3bar v\3bar_{2} \leq \alpha_2 \|v\|_{2,h}
\end{equation}
for all $v\in W_{h,k}$.
\end{lemma}

\begin{proof} Note that, for any $\phi\in S_k(T)$, we have
$$
({\cal Q}_h (a_{ij}\partial_{ij,d}^2 v), \phi)_T =
(\partial_{ij,d}^2 v, {\cal Q}_h (a_{ij} \phi))_T.
$$
With $\varphi= {\cal Q}_h (a_{ij} \phi)$, we have
$$
(\partial^2_{ij}v_0, \varphi )_T = (\partial^2_{ij}v_0,{\cal Q}_h
(a_{ij} \phi))_T =({\cal Q}_h (a_{ij} \partial^2_{ij}v_0), \phi)_T.
$$
Thus, using (\ref{2.4new}) we arrive at
\begin{equation}\label{EQ:800:01}
\begin{split}
({\cal Q}_h (a_{ij}\partial_{ij,d}^2 v), \phi)_T = &
(\partial_{ij,d}^2 v, a_{ij} \phi)_T = (\partial_{ij,d}^2 v, \varphi)_T\\
= & (\partial^2
 _{ij}v_0, \varphi )_T-\langle (v_b-v_0) n_i,\partial_j\varphi\rangle_{\partial
 T}\\
 & +\langle v_{gi}-\partial_i v_0, n_j\varphi\rangle_{\partial
 T}\\
= & ({\cal Q}_h (a_{ij} \partial^2_{ij}v_0), \phi)_T-\langle
(v_b-v_0) n_i,\partial_j\varphi\rangle_{\pT}\\
 & +\langle v_{gi}-\partial_i v_0, n_j\varphi\rangle_{\pT}.
 \end{split}
 \end{equation}
It now follows from the Cauchy-Schwarz and the trace inequality
(\ref{x}) that
\begin{equation}\label{EQ:800:02}
\begin{split}
 \left|({\cal Q}_h (a_{ij}\partial_{ij,d}^2 v), \phi)_T \right|
 \leq\ & \|{\cal Q}_h (a_{ij}\partial_{ij}^2 v_0)\|_T \|\phi\|_T +
 \|v_b-v_0\|_{\pT} \|\partial_j\varphi\|_\pT \\
 & \ + \|v_{gi}-\partial_i v_0\|_{\pT} \|\varphi\|_{\pT}\\
 \leq   &  \|{\cal Q}_h (a_{ij}\partial_{ij}^2 v_0)\|_T \|\phi\|_T+
 C h_T^{-\frac32} \|v_b-v_0\|_{\pT} \|\varphi\|_T \\
 & \ + Ch_T^{-\frac12}\|v_{gi}-\partial_i v_0\|_{\pT} \|\varphi\|_T.
\end{split}
 \end{equation}
It is easy to see that $\|\varphi\|_T \leq C \|\phi\|_T$. Thus, by
choosing $\phi ={\cal Q}_h (a_{ij}\partial_{ij,d}^2 v)$ in
(\ref{EQ:800:02}) we obtain
\begin{equation*}
\|{\cal Q}_h (a_{ij}\partial_{ij,d}^2 v)\|_T^2 \leq C \left(\|{\cal
Q}_h (a_{ij}\partial_{ij}^2 v_0)\|_T^2 +
 h_T^{-3} \|v_b-v_0\|_{\pT}^2+ h_T^{-1}\|v_{gi}-\partial_i
 v_0\|_{\pT}^2\right),
 \end{equation*}
which, after summing over all $T\in\T_h$, gives the upper-bound
estimate of $\3bar v\3bar_2$ in (\ref{EQ:Sept:2015:001}). The
lower-bound estimate of $\3bar v\3bar_2$ can be established in a
similar manner by representing $ ({\cal Q}_h (a_{ij}
\partial^2_{ij}v_0), \phi)_T$ in terms of $({\cal Q}_h
(a_{ij}\partial_{ij,d}^2 v), \phi)_T$ and other two boundary
integrals in (\ref{EQ:800:01}). This completes the proof of the
lemma.
\end{proof}

\begin{lemma}\label{lem3}
(inf-sup condition) Assume that the coefficient matrix
$a=\{a_{ij}\}_{d\times d}$ is uniformly piecewise continuous in
$\Omega$ with respect to the finite element partition $\T_h$. For
any $\sigma\in S_{h,k}$, there exists $v_\sigma\in W_{h,k}^0$
satisfying
\begin{eqnarray} \label{EQ:August30:500}
b_h(v_\sigma,\sigma) & \ge & \frac12 \|\sigma\|_{0}^2,\\
\| v_\sigma\|^2_{2,h} &\leq & C \|\sigma\|^2_{0},
\label{EQ:August30:501}
\end{eqnarray}
provided that the meshsize $h<h_0$ for a sufficiently small, but
fixed parameter $h_0>0$.
\end{lemma}

\begin{proof} Consider the following second order elliptic problem:
\begin{align}\label{pro1}
\sum_{i,j=1}^d a_{ij}\partial_{ij}^2 w=&\ \sigma,  \qquad \text{in}\ \Omega,\\
w=&\ 0,\qquad \text{on}\  \partial\Omega.\label{pro1-2}
\end{align}
By the $H^{2}$-regularity assumption
(\ref{Assumption:H2regularity}), the problem
(\ref{pro1})-(\ref{pro1-2}) has a unique solution in $H^{2}(\Omega)$
satisfying
\begin{equation}\label{regu}
\|w\|_{2}\leq C\|\sigma\|_{0}.
\end{equation}

We claim that $v_\sigma=Q_hw$ satisfies
(\ref{EQ:August30:500})-(\ref{EQ:August30:501}). In fact, by setting
$v=v_\sigma=Q_hw$ in $b_h(v,\sigma)$, we have from the commutative
property (\ref{EQ:CommutativeP}), the equation (\ref{pro1}), and the
a priori estimate (\ref{regu}) that
\begin{equation}\label{bv}
\begin{split}
b_h(v_\sigma,\sigma)=& \sum_{T\in {\cal T}_h} (\sum_{i,j=1}^d
a_{ij}\partial_{ij,d}^2 Q_hw,\sigma)_T\\
=& \sum_{T\in {\cal T}_h}(\sum_{i,j=1}^d
a_{ij}{\cal Q}_h\partial_{ij}^2 w,\sigma )_T\\
=&\sum_{T\in {\cal
T}_h}(\sum_{i,j=1}^da_{ij}\partial_{ij}^2w,\sigma)_T+ \sum_{T\in
{\cal T}_h}(\sum_{i,j=1}^da_{ij}({\cal Q}_h-I)\partial_{ij}^2
w,\sigma)_T\\
=&\sum_{T\in {\cal T}_h}\|\sigma\|^2_{T}+
\sum_{T\in {\cal T}_h}\sum_{i,j=1}^d(({\cal Q}_h-I)\partial_{ij}^2 w, (a_{ij}-\bar{a}_{ij})\sigma)_T\\
\geq & \ \|\sigma\|^2_{0}-\varepsilon(h) \|w\|_{2}\|\sigma\|_{0}\\
\geq & \ (1-C \varepsilon(h))\|\sigma\|^2_{0},
\end{split}
\end{equation}
where $\varepsilon(h)$ is given by $\|a_{ij}
-\bar{a}_{ij}\|_{L^\infty(\Omega)}$. Since $a_{ij}$ is uniformly
piecewise continuous, there exists a small, but fix $h_0$, such that
$1-C\varepsilon(h) \ge \frac12$ when $h<h_0$. It follows that
$$
b_h(v_\sigma,\sigma) \geq \frac12 \|\sigma\|^2_{0},
$$
which verifies the inequality (\ref{EQ:August30:500}).

Next, for the same $v_\sigma=Q_hw$, from the commutative property
(\ref{EQ:CommutativeP}) and the stability of the $L^2$ projection
${\cal Q}_h$, we have
\begin{equation}\label{EQ:August31:000}
\begin{split}
\sum_{T\in {\cal T}_h} \|{\cal
Q}_h(\sum_{i,j=1}^da_{ij}\partial_{ij,d}^2 v_\sigma)\|_{T}^2\leq& C
\sum_{T\in {\cal T}_h} \sum_{i,j=1}^d\|a_{ij}\partial_{ij,d}^2
Q_hw \|_{T}^2\\
=&C\sum_{T\in {\cal T}_h} \sum_{i,j=1}^d\|a_{ij}{\cal Q}_h \partial^2_{ij} w\|_{T}^2\\
\leq & C \|w\|^2_{2}  \leq C\|\sigma\|^2_{0}.
\end{split}
\end{equation}

For $v=Q_hw$, by the trace inequality (\ref{trace-inequality}) and
(\ref{regu}), the estimate (\ref{3.2}) with $m=1$, we have
\begin{equation}\label{EQ:August31:001}
\begin{split}
\sum_{T\in\T_h}h_T^{-3}\|v_0-v_b\|^2_{\pT}
= &\ \sum_{T\in\T_h} h_T^{-3}\|Q_0w-Q_bw\|^2_{\pT}\\
\leq & \  \sum_{T\in\T_h} h_T^{-3} \|Q_0w-w\|^2_{\pT} \\
\leq &\ C\sum_{T\in {\cal T}_h} h_T^{-4}\left(\|Q_0w-w\|^2_{T}
+h_T^{2}\|\nabla Q_0w-\nabla w\|^2_{T} \right)\\
\leq &\ C \|w\|_{2}^2 \leq  C\|\sigma\|^2_{0}.
 \end{split}
\end{equation}
A similar argument can be applied to yield the following estimate
\begin{equation}\label{EQ:August31:002}
\sum_{T\in {\cal T}_h }h_T^{-1}\|\nabla v_0-\textbf{v}_g\|^2_{\pT}
\leq C\|\sigma\|^2_{0}.
\end{equation}
Now combining (\ref{EQ:August31:000}) with (\ref{EQ:August31:001})
and (\ref{EQ:August31:002}) gives $\3bar v\3bar^2_{2} \leq
C\|\sigma\|^2_{0}$, and hence from (\ref{EQ:Sept:2015:001}) we
obtain
$$
\| v_\sigma\|^2_{2,h} \leq C\|\sigma\|^2_{0},
$$
which, together with (\ref{bv}), completes the proof of the lemma.
\end{proof}

\begin{lemma}\label{lem1} (boundedness) The following
inequalities hold true
\begin{align}\label{b1}
|s_h(u,v)|\leq & \ \| u\|_{2,h} \| v \|_{2,h},\qquad \forall u,v \in W_{h,k}^0,\\
|b_h(v,\sigma)|\leq & \ C \| v\|_{2,h} \|\sigma\|_0, \qquad  \forall
v \in W_{h,k}^0, \ \sigma\in S_{h,k}.\label{b2}
\end{align}
\end{lemma}

\begin{proof}
To derive (\ref{b1}), we use the Cauchy-Schwarz inequality to obtain
\begin{equation*}
\begin{split}
|s_h(u,v)| =&\Big|\sum_{T\in {\cal T}_h} h_T^{-3} \langle u_0-u_b,
v_0-v_b\rangle_\pT+ h_T^{-1} \langle \nabla u_0 -\bu_g, \nabla v_0 -\bv_g \rangle_\pT \Big|\\
\leq &  \Big(\sum_{T\in {\cal T}_h}h_T^{-3}\|u_0-u_b\|^2_{\partial
T}\Big)^{\frac{1}{2}}\Big(\sum_{T\in {\cal T}_h} h_T^{-3}
\|v_0-v_b\|^2_{\partial
T}\Big)^{\frac{1}{2}}\\
&+  \Big(\sum_{T\in {\cal T}_h} h_T^{-1} \|\nabla u_0 -\textbf{u}_g
\|^2_{\partial T}\Big)^{\frac{1}{2}}\Big(\sum_{T\in {\cal T}_h}
h_T^{-1} \|\nabla v_0 -\textbf{v}_g \|^2_{\partial
T}\Big)^{\frac{1}{2}}\\
 \leq&  \| u\|_{2,h} \| v\|_{2,h}.
\end{split}
\end{equation*}

As to (\ref{b2}), by the definition of ${\cal  Q}_h$ and the
Cauchy-Schwarz inequality, for any $v\in W_{h,k}^0$ and $\sigma\in
S_h$, we have
\begin{equation*}
\begin{split}
 b_h(v,\sigma) & =  \sum_{T\in {\cal T}_h}( \sum_{i,j=1}^d a_{ij}
 \partial_{ij,d}^2 v, \sigma )_T \\
 & =  \sum_{T\in {\cal T}_h}  \sum_{i,j=1}^d ({\cal  Q}_h(a_{ij}
 \partial_{ij,d}^2 v), \sigma )_T \\
&\leq \Big(\sum_{T\in {\cal T}_h} \|  \sum_{i,j=1}^d{\cal Q}_h
(a_{ij} \partial_{ij,d}^2 v)
\|_T^2\Big)^{\frac{1}{2}}\Big(\sum_{T\in {\cal
T}_h}\|\sigma\|_T^2\Big)^{\frac{1}{2}}\\
&\leq \3bar v\3bar_2 \|\sigma\|_0.\\
\end{split}
\end{equation*}
This, along with (\ref{EQ:Sept:2015:001}), completes the proof.
\end{proof}

Introduce the following subspace of $W_{h,k}^0$:
$$
Z_h=\{v\in W_{h,k}^0: \ b_h(v,\sigma)=0, \ \forall \sigma\in
S_{h,k}\}.
$$

\begin{lemma}\label{lem2} (coercivity) There exists a constant $\alpha>0$ such that
\begin{equation}\label{EQ:coercivity}
s_h(v,v)\geq \alpha \| v\|_{2,h}^2, \qquad \forall v\in Z_h.
\end{equation}
\end{lemma}

\begin{proof}
Given any $v\in Z_h$, we have $b_h(v,\sigma)=0$ for all $\sigma\in
S_{h,k}$. Using (\ref{EQ:global-b-form}) and (\ref{EQ:local-b-form})
we obtain
\begin{eqnarray*}
0&=&b_h(v,\sigma)\\
&=&\sum_{T\in {\cal T}_h}(\sum_{i,j=1}^d a_{ij}\partial_{ij,d}^2 v ,\sigma)_T\\
&=&\sum_{T\in {\cal T}_h}(\sum_{i,j=1}^d{\cal Q}_h
(a_{ij}\partial_{ij,d}^2 v),\sigma)_T
\end{eqnarray*}
for all $\sigma\in S_{h,k}$. Thus, on each element $T\in\T_h$ we
have
$$
\sum_{i,j=1}^d{\cal  Q}_h (a_{ij}\partial_{ij,d}^2 v) = 0.
$$
It follows that $\3bar v\3bar_2^2 = s_h(v,v),$ which, together with
(\ref{EQ:Sept:2015:001}), implies the desired coercivity
(\ref{EQ:coercivity}) for some $\alpha>0$.
\end{proof}

Using the abstract theory for saddle-point problems developed by
Babu\u{s}ka \cite{babuska} and Brezzi \cite{b1974}, we arrive at the
following result.

\begin{theorem}\label{thmunique1} Assume that the coefficient functions $a_{ij}$
are uniformly piecewise continuous in $\Omega$ with respect to the
finite element partition $\T_h$. The primal-dual weak Galerkin
finite element scheme (\ref{2})-(\ref{32}) has a unique solution
$(u_h; \lambda_h)\in W_{h,k}^0\times S_{h,k}$, provided that the
meshsize $h<h_0$ holds true for a sufficiently small, but fixed
parameter value $h_0>0$. Moreover, there exists a constant $C$ such
that the solution $u_h$ and $\lambda_h$ satisfies
$$
\| u_h\|_{2,h} + \|\lambda_h\|_0 \leq C \|f\|_0.
$$
\end{theorem}

\section{Error Estimates}\label{Section:ErrorEstimates}
Let $(u_h;\lambda_h) \in W_{h,k}^0\times S_{h,k}$ be the approximate solution
of the problem (\ref{1}) arising from the primal-dual weak Galerkin
finite element scheme (\ref{2})-(\ref{32}). Note that $\lambda=0$ is
the solution of the trivial dual problem of $b(v,\lambda)=0$ for all
$v\in H^2(\Omega)\cap H_0^1(\Omega)$. Define the error functions by
\begin{equation}\label{error}
e_h=u_h-Q_hu,\quad \gamma_h =\lambda_h-\Q_h\lambda,
\end{equation}
where $Q_h$ and $\Q_h$ are the corresponding $L^2$ projection
operators.

\begin{lemma}\label{errorequa}
The error functions $e_h$ and $\gamma_h$ given by (\ref{error})
satisfy the following equations
\begin{eqnarray}\label{sehv}
s_h(e_h, v)+b_h(v, \gamma_h) & = &
  -s_h(Q_h u,v),\qquad \forall v\in W_{h,k}^0,\\
b_h(e_h,\sigma) & = & \ell_u(\sigma), \qquad\;\; \qquad \forall
\sigma\in S_{h,k},\label{sehv2}
\end{eqnarray}
where
\begin{equation}\label{EQ:ell-u}
\ell_u(\sigma) = \sum_{T\in {\cal T}_h} \sum_{i,j=1}^d((I-{\cal
Q}_h)\partial_{ij}^2u, a_{ij}\sigma)_T.
\end{equation}
\end{lemma}

\begin{proof}
First, by subtracting $s_h(Q_h u,v)$ from both sides of (\ref{2}) we
obtain
\begin{align*}
s_h(u_h-Q_h u,v)+b_h(v, \lambda_h) = -s_h(Q_h u,v),\qquad\forall
v\in W_{h,k}^0.
\end{align*}
It follows from $\lambda=0$ that $\gamma_h=\lambda_h$. Thus, the
above equation can be rewritten as
\begin{equation}\label{EQ:Error-EQ-01}
s_h(e_h,v)+b_h(v, \gamma_h) = -s_h(Q_h u,v),\qquad\forall v\in
W_{h,k}^0,
\end{equation}
which is the first error equation (\ref{sehv}).

To derive (\ref{sehv2}), we use (\ref{1}) and
(\ref{EQ:CommutativeP}) in Lemma \ref{Lemma5.1} to obtain
\begin{equation*}
\begin{split}
b_h(Q_h u, \sigma) = &\sum_{T\in {\cal T}_h} (\sum_{i,j=1}^d
a_{ij}\partial_{ij,d}^2
Q_hu,\sigma)_T\\
=&\sum_{T\in {\cal T}_h} (\sum_{i,j=1}^d a_{ij}{\cal Q}_h\partial_{ij }^2
u,\sigma)_T\\
=&\sum_{T\in {\cal T}_h}(\sum_{i,j=1}^d a_{ij} \partial_{ij }^2 u,\sigma)_T+
\sum_{T\in {\cal T}_h} (\sum_{i,j=1}^da_{ij}({\cal Q}_h-I)\partial_{ij}^2 u, \sigma)_T\\
=& ( f,\sigma) +\sum_{T\in {\cal T}_h} \sum_{i,j=1}^d(({\cal
Q}_h-I)\partial_{ij}^2u, a_{ij}\sigma)_T,
\end{split}
\end{equation*}
for all $\sigma\in S_{h,k}$. Now subtracting the above equation from
(\ref{32}) yields the desired equation (\ref{sehv2}). This completes
the proof of the lemma.
\end{proof}

The equations (\ref{sehv}) and (\ref{sehv2}) are called {\em error
equations} for the primal-dual WG finite element scheme
(\ref{2})-(\ref{32}). This is a saddle point system for which the
Brezzi's Theorem \cite{b1974} can be applied for a stability
analysis.

Recall that $\T_h$ is a shape-regular finite element partition of
the domain $\Omega$. For any $T\in\T_h$ and $\varphi\in H^1(T)$, the
following trace inequality holds true \cite{wy3655}:
\begin{equation}\label{trace-inequality}
\|\varphi\|_{\pT}^2 \leq C
(h_T^{-1}\|\varphi\|_{T}^2+h_T\|\nabla\varphi\|_{T}^2).
\end{equation}
If $\varphi$ is a polynomial on the element $T\in \T_h$, then from
the inverse inequality (see also \cite{wy3655}) we have
\begin{equation}\label{x}
\|\varphi\|_{\pT}^2 \leq C h_T^{-1}\|\varphi\|_{T}^2.
\end{equation}

The following estimates for the $L^2$-projections are extremely
useful in the forthcoming error analysis.

\begin{lemma}\label{Lemma5.2}\cite{wy3655}  Let ${\cal T}_h$ be a
finite element partition of $\Omega$ satisfying the shape regularity
assumptions given in \cite{wy3655}. Then, for any $0\leq s\leq 2$
and $1\leq m\leq k$, one has
\begin{eqnarray}\label{3.2}
\sum_{T\in {\cal T}_h}h_T^{2s}\|u-Q_0u\|^2_{s,T} &\leq &
Ch^{2(m+1)}\|u\|_{m+1}^2,\\
\label{3.3-3} \sum_{T\in {\cal T}_h}\sum_{i,j=1}^dh_T^{2s}\| u-{\cal
Q}_h u\|^2_{s,T} &\leq& Ch^{2(m-1)}\|u\|_{m-1}^2,\\
\label{3.3} \sum_{T\in {\cal
T}_h}\sum_{i,j=1}^dh_T^{2s}\|\partial^2_{ij}u-{\cal
Q}_h\partial^2_{ij}u\|^2_{s,T} &\leq& Ch^{2(m-1)}\|u\|_{m+1}^2.
\end{eqnarray}
\end{lemma}

\begin{theorem} \label{theoestimate} Assume that the coefficient functions $a_{ij}$
are uniformly piecewise continuous in $\Omega$ with respect to the
finite element partition $\T_h$. Let $u$ and $(u_h;\lambda_h) \in
W_{h,k}^0\times S_{h,k}$ be the solutions of (\ref{1}) and
(\ref{2})-(\ref{32}), respectively. Assume that the exact solution
$u$ of (\ref{1}) is sufficiently regular such that $u\in
H^{k+1}(\Omega)$. There exists a constant $C$ such that
 \begin{equation}\label{erres}
\| u_h-Q_h u \|_{2,h}+\|\lambda_h-{\cal Q}_h \lambda\|_0 \leq
Ch^{k-1}\|u\|_{k+1},
\end{equation}
provided that the meshsize $h<h_0$ holds true for a sufficiently
small, but fixed $h_0>0$.
\end{theorem}

\begin{proof}  It follows from Lemma \ref{lem3}, Lemma \ref{lem1},
and Lemma \ref{lem2} that the Brezzi's stability conditions are
satisfied for the saddle point system (\ref{sehv})-(\ref{sehv2}).
Thus, there exists a constant $C$ such that
\begin{equation}\label{ess2}
\| e_h \|_{2,h}+\|\gamma_h\|_0 \leq C\left( \sup_{v\in W_{h,k}^0,
v\neq 0} \frac{|s_h(Q_hu, v)|}{\| v\|_{2,h}}+ \sup_{\sigma\in S_h,
\sigma\neq 0} \frac{|\ell_u(\sigma)|}{\|\sigma\|_0} \right).
\end{equation}

Recall that
\begin{equation}\label{EQ:September:03:100}
\begin{split}
s_h(Q_h u, v) = &\sum_{T\in {\cal T}_h}h_T^{-3}\langle Q_0u-Q_bu,
v_0-v_b\rangle_\pT\\
& +\sum_{T\in {\cal T}_h}h_T^{-1} \langle \nabla Q_0u
-\textbf{Q}_g(\nabla u), \nabla v_0 -\textbf{v}_g\rangle_\pT.
\end{split}
\end{equation}
The first term on the right-hand side of (\ref{EQ:September:03:100})
can be estimated by using the Cauchy-Schwarz inequality, the trace
inequality (\ref{trace-inequality}), and the estimate (\ref{3.2})
with $m=k$ as follows
\begin{equation}\label{s1}
\begin{split}
&\left| \sum_{T\in {\cal T}_h}h_T^{-3}\langle Q_0u-Q_bu, v_0-v_b\rangle_\pT\right|\\
= &\left| \sum_{T\in {\cal T}_h}h_T^{-3}\langle Q_0u- u, v_0-v_b\rangle_\pT\right|\\
 \leq & \Big(\sum_{T\in {\cal T}_h}h_T^{-3}\|u-Q_0u\|^2_{\partial
T}\Big)^{\frac{1}{2}} \Big(\sum_{T\in {\cal T}_h}
h_T^{-3}\|v_0-v_b\|^2_{\partial T}\Big)^{\frac{1}{2}}\\
\leq & C\Big(\sum_{T\in {\cal
T}_h}h_T^{-4}\big(\|u-Q_0u\|_T^2+h_T^2\|u-Q_0u\|_{1,T}^2\big)
\Big)^{\frac{1}{2}}\| v\|_{2,h}\\
\leq & Ch^{k-1}\|u\|_{k+1}\|v\|_{2,h}.
\end{split}
\end{equation}
Similarly, the second term on the right-hand side of
(\ref{EQ:September:03:100}) has the following estimate
\begin{equation}\label{s2}
\left|\sum_{T\in {\cal T}_h}h_T^{-1}\langle \nabla Q_0u
-\textbf{Q}_g(\nabla u ), \nabla v_0 -\textbf{v}_g \rangle_\pT
\right|\leq Ch^{k-1}\|u\|_{k+1} \| v\|_{2,h}.
\end{equation}
Combining (\ref{EQ:September:03:100}) with (\ref{s1}) and (\ref{s2})
gives
\begin{equation}\label{EQ:True:01}
|s_h(Q_h u, v)| \leq C h^{k-1} \|u\|_{k+1} \| v\|_{2,h}.
\end{equation}

As to the second term on the right-hand side of (\ref{ess2}), using
(\ref{EQ:ell-u}) and the estimate (\ref{3.3})  with $m=k$ we have
\begin{equation}\label{aij}
\begin{split}
|\ell_u(\sigma)| & = \left|\sum_{T\in\T_h} \sum_{i,j=1}^d (I-{\cal
Q}_h)\partial_{ij}^2u, a_{ij}\sigma)_T\right|\\
& \leq \sum_{i,j=1}^d \|a_{ij}\|_{L^\infty}\ \|(I-{\cal
Q}_h)\partial_{ij}^2u\|_0 \ \|\sigma\|_0 \\
&\leq C h^{k-1}\|u\|_{k+1} \|\sigma\|_0.
\end{split}
\end{equation}
Substituting (\ref{EQ:True:01}) and (\ref{aij}) into (\ref{ess2})
gives the error estimate (\ref{erres}).
\end{proof}

\section{Error Estimates in $H^1$ and
$L^2$}\label{Section:H1L2Error}
We first establish an estimate for the discrete weak second order
partial derivatives.
\begin{lemma}\label{lemma7.2} There exists a constant $C$ such that
for any $v\in W_{k}(T)$, we have
\begin{equation}\label{qaij}
\|\partial^2_{ij,d} v\|_T^2 \leq C\left(\|\partial_{ij}^2 v_0\|_T^2
+ s_T(v,v)\right),
\end{equation}
where $C$ is a generic constant independent of $T\in\T_h$.
\end{lemma}

\begin{proof}
From (\ref{2.4new}), for any $\varphi\in S_k(T)$, we have
\begin{equation*}\label{7.22}
(\partial_{ij,d}^2 v, \varphi)_T = (\partial^2_{ij}v_0
,\varphi)_T-\langle v_b-v_0, \partial_j\varphi n_i\rangle_{\pT}
+\langle v_{gi}-\partial_i v_0,\varphi n_j\rangle_{\pT}.
\end{equation*}
Using the Cauchy-Schwarz inequality, the trace inequality (\ref{x}),
and the inverse inequality we arrive at
\begin{equation*}\label{7.22new}
\begin{split}
|(\partial_{ij,d}^2 v, \varphi)_T| \leq & \|\partial^2_{ij}v_0\|_T
\|\varphi\|_T + \|v_b-v_0\|_\pT \|\partial_j\varphi\|_{\pT}
+\|v_{gi}-\partial_i v_0\|_\pT \|\varphi\|_{\pT}\\
\leq & \left(\|\partial^2_{ij}v_0\|_T +
Ch_T^{-\frac32}\|v_b-v_0\|_\pT + Ch_T^{-\frac12}\|v_{gi}-\partial_i
v_0\|_\pT\right)\|\varphi\|_T.
\end{split}
\end{equation*}
Thus,
$$
\|\partial_{ij,d}^2 v\|_T^2\leq C\left( \|\partial^2_{ij}v_0\|_T^2 +
h_T^{-3}\|v_b-v_0\|_\pT^2 + h_T^{-1}\|v_{gi}-\partial_i v_0\|_\pT^2
\right),
$$
which verifies the inequality (\ref{qaij}). This completes the proof
of the Lemma.
\end{proof}

\medskip

Consider the problem of solving an unknown function $w$ such that
\begin{align}\label{dual1}
\sum_{i,j=1}^d \partial_{ji}^2 (a_{ij} w)=& \ \theta,\qquad \text{in}\ \Omega,\\
w=& \ 0,\qquad \text{on}\ \partial\Omega, \label{dual2}
\end{align}
where $\theta$ is a given function. With the bilinear form
$b(\cdot,\cdot)$ given by (\ref{b-form}), a variational formulation
for (\ref{dual1})-(\ref{dual2}) reads as follows: Find $w\in
L^2(\Omega)$ such that
\begin{equation}\label{Dual-Variational}
b(v, w) = (\theta, v)\qquad \forall v\in H^2(\Omega)\cap
H_0^1(\Omega).
\end{equation}
The problem (\ref{dual1})-(\ref{dual2}) is said to be
$H^{1+s}$-regular, $s\in [0,1]$, if for any $\theta\in
H^{s-1}(\Omega)$, there exists a unique $w\in H^{1+s}(\Omega)\cap
H_0^1(\Omega)$ satisfying (\ref{Dual-Variational}) and the following
a priori estimate:
\begin{equation}\label{regul}
\|w\|_{1+s}\leq C\|\theta\|_{s-1}.
\end{equation}

\begin{lemma}\label{Lemma:TechnicalEquality}
Assume that the coefficients $a_{ij}$ are in $C^1(\Omega)$. Then,
for any $v=\{v_0, v_b, \bv_g\}\in W_{h,k}^0$, the following identity
holds true
\begin{equation}\label{2.14:800}
\begin{split}
(v_0, \theta) =&\sum_{T\in{\cal T}_h} \sum_{i,j=1}^d
(a_{ij}\partial^2_{ij,d} v,
w)_T-\langle(v_{gi} -\partial_i v_0) n_j, ({\cal Q}_h-I)(a_{ij}w)\rangle_{\partial T}\\
&+ \langle (v_b-v_0) n_i,\partial_j( {\cal
Q}_h-I)(a_{ij}w)\rangle_{\partial T}.
\end{split}
\end{equation}

\end{lemma}

\begin{proof} By testing (\ref{dual1}) with $v_0$ on each element $T\in\T_h$,
we obtain from the usual integration by parts
\begin{equation}\label{2.11}
\begin{split}
(\theta, v_0)
=&\sum_{T\in{\cal T}_h}(\sum_{i,j=1}^d \partial_{ji}^2 (a_{ij}w), v_0)_T\\
=&\sum_{T\in{\cal T}_h} \sum_{i,j=1}^d (a_{ij}w,\partial_{ij}^2
v_0)_T-\langle a_{ij}w n_j, \partial_i v_0\rangle_{\partial
T}+\langle\partial_j (a_{ij}w), v_0 n_i\rangle_{\partial T}\\
=&\sum_{T\in{\cal T}_h} \sum_{i,j=1}^d (a_{ij}w,\partial_{ij}^2
v_0)_T-\langle a_{ij}w n_j, \partial_i v_0-
v_{gi}\rangle_{\partial T}\\
&+\langle \partial_j (a_{ij}w)
 n_i, v_0-v_b \rangle_{\partial T},\\
\end{split}
\end{equation}
where we have used the homogeneous boundary condition (\ref{dual2})
in the third line and the fact that $a_{ij}\in C^1(\Omega)$ and
$v_b=0$ on $\partial\Omega$ in the fourth line.

From (\ref{2.4new}) with $\varphi={\cal Q}_h(a_{ij}w)$, we have
\begin{align*}
(\partial^2_{ij,d}v, {\cal Q}_h(a_{ij}w))_T = &(\partial^2_{ij}
v_0,{\cal Q}_h(a_{ij}w))_T - \langle v_b-v_0, n_i\partial_j {\cal
Q}_h(a_{ij}w)\rangle_{\pT} \\
& +\langle v_{gi} -\partial_i v_0, n_j{\cal Q}_h(a_{ij}w)\rangle_{\pT}\\
=&(\partial^2_{ij}v_0,a_{ij}w)_T - \langle v_b-v_0, n_i\partial_j
{\cal Q}_h(a_{ij}w)\rangle_{\pT} \\
&+\langle v_{gi} -\partial_i v_0, n_j{\cal
Q}_h(a_{ij}w)\rangle_{\pT},
\end{align*}
which leads to
\begin{equation}\label{2.13}
\begin{split}
&(\partial_{ij}^2 v_0, a_{ij}w)_T =(\partial^2_{ij,d} v, {\cal
Q}_h(a_{ij}w))_T
\\&\qquad -\langle v_{gi} -\partial_i v_0, n_j{\cal Q}_h(a_{ij}w)
\rangle_{\partial T}+\langle v_b-v_0, n_i\partial_j {\cal
Q}_h(a_{ij}w)\rangle_{\partial T}.
\end{split}
\end{equation}
Using (\ref{2.13}), we can rewrite (\ref{2.11}) as follows
\begin{equation}\label{2.14}
\begin{split}
(v_0, \theta) =&\sum_{T\in{\cal T}_h}\sum_{i,j=1}^d
(\partial^2_{ij,d} v, {\cal Q}_h(a_{ij}w))_T-\langle v_{gi}
-\partial_i v_0, n_j{\cal Q}_h(a_{ij}w) \rangle_{\partial T}\\ &+
\langle v_b-v_0, n_i\partial_j {\cal Q}_h(a_{ij}w)\rangle_{\partial
T}-\langle a_{ij} w n_j, \partial_i v_0- v_{gi}\rangle_{\partial
T}\\
&+\langle n_i \partial_j (a_{ij}w), v_0-v_b\rangle_{\partial T}\\
=&\sum_{T\in{\cal T}_h} \sum_{i,j=1}^d (a_{ij}\partial^2_{ij,d} v,
w)_T-\langle(v_{gi} -\partial_i v_0) n_j, ({\cal Q}_h-I)(a_{ij}w)\rangle_{\partial T}\\
&+ \langle (v_b-v_0) n_i,\partial_j( {\cal
Q}_h-I)(a_{ij}w)\rangle_{\partial T},
\end{split}
\end{equation}
which is the desired identity (\ref{2.14:800}).
\end{proof}

The following Lemma is developed for an estimate of the last two
terms on the right-hand side of (\ref{2.14:800}) with the
$H^1$-regularity assumption for the dual problem
(\ref{Dual-Variational}).

\begin{lemma}\label{Lemma:TechnicalEstimates:01}
Assume that the coefficient matrix $\{a_{ij}\}_{d\times d}$ is
regular so that $a_{ij}\in \Pi_{T\in\T_h} W^{1,\infty}(T)$. Then,
there exists a constant $C$ such that for any $v\in W_{h,k}^0$, we
have
\begin{eqnarray}\label{2.14.100:10}
\left|\sum_{T\in{\cal T}_h} \sum_{i,j=1}^d  \langle(v_{gi}
-\partial_i v_0) n_j, ({\cal Q}_h-I)(a_{ij}w)\rangle_{\partial T}
\right| &\leq & Ch \ \|v\|_{2,h} \|\theta\|_{-1},
\\ \label{2.14.110:10} \left|\sum_{T\in{\cal T}_h} \sum_{i,j=1}^d
\langle(v_{b} -v_0) n_i,
\partial_j({\cal Q}_h-I)(a_{ij}w)\rangle_{\partial T}
\right| &\leq & Ch \  \|v\|_{2,h}\|\theta\|_{-1},
\end{eqnarray}
provided that the dual problem (\ref{Dual-Variational}) has the
$H^1$-regularity estimate (\ref{regul}) with $s=0$.
\end{lemma}

\begin{proof} We only present a proof for the inequality
(\ref{2.14.100:10}), as (\ref{2.14.110:10}) can be derived in a
similar way. From the Cauchy-Schwarz inequality, the trace
inequality (\ref{trace-inequality}), and the estimates in Lemma
\ref{Lemma5.2} we have
\begin{equation}\label{2.14.100}
\begin{split}
& \left|\sum_{T\in{\cal T}_h} \sum_{i,j=1}^d  \langle(v_{gi}
-\partial_i v_0) n_j, ({\cal Q}_h-I)(a_{ij}w)\rangle_{\partial T}
\right| \\
\leq & \sum_{T\in{\cal T}_h} \sum_{i,j=1}^d \|v_{gi} -\partial_i
v_0\|_\pT \|({\cal Q}_h-I)(a_{ij}w)\|_\pT\\
\leq & \ C \left( \sum_{T\in{\cal T}_h} \sum_{i,j=1}^d h_T \|({\cal
Q}_h-I)(a_{ij}w)\|_\pT^2\right)^{\frac12} \|v\|_{2,h}\\
\leq & \ Ch \|w\|_1 \|v\|_{2,h}  \leq   Ch \|\theta\|_{-1}
\|v\|_{2,h},
\end{split}
\end{equation}
where we have used the $H^1$-regularity assumption in the last line.
This completes the proof of the lemma.
\end{proof}

Note that if $P_1(T) \subseteq S_k(T)$ for all $T\in\T_h$ and
$a_{ij}\in \Pi_{T\in\T_h} W^{2,\infty}(T)$, then from the trace
inequality (\ref{trace-inequality}) and the standard error estimate
for the $L^2$ projection ${\cal Q}_h$ we have
\begin{equation}
\begin{split}
\|({\cal Q}_h - I)(a_{ij}w)\|_\pT^2 \leq & C h_T^{-1}(\|({\cal Q}_h
- I)(a_{ij}w)\|_T^2 + h_T^2 \|({\cal Q}_h -
I)(a_{ij}w)\|_{1,T}^2)\\
\leq & C h_T^3 \|a_{ij}\|_{2,\infty,T}^2 \|w\|_{2,T}^2.
\end{split}
\end{equation}
By substituting the above inequality into the third line of
(\ref{2.14.100}) and then assuming the $H^2$-regularity
(\ref{regul}) we obtain the following result.

\begin{lemma}\label{Lemma:TechnicalEstimates:02}
Assume that the coefficients $a_{ij}$ are sufficient smooth on each
element such that $a_{ij}\in \Pi_{T\in\T_h} W^{2,\infty}(T)$. In
addition, assume $P_1(T)\subset S_k(T)$ for each element $T\in\T_h$.
Then, there exists a constant $C$ such that for any $v\in
W_{h,k}^0$, we have
\begin{eqnarray}\label{2.14.100:12}
\left|\sum_{T\in{\cal T}_h} \sum_{i,j=1}^d  \langle(v_{gi}
-\partial_i v_0) n_j, ({\cal Q}_h-I)(a_{ij}w)\rangle_{\partial T}
\right| & \leq & Ch^2 \ \|v\|_{2,h}\|\theta\|_{0},\\
\label{2.14.110:15} \left|\sum_{T\in{\cal T}_h} \sum_{i,j=1}^d
\langle(v_{b} -v_0) n_i,
\partial_j({\cal Q}_h-I)(a_{ij}w)\rangle_{\partial T}
\right|  &\leq & Ch^2 \  \|v\|_{2,h}\|\theta\|_{0},
\end{eqnarray}
provided that the regularity estimate (\ref{regul}) holds true with
$s=1$.
\end{lemma}

\begin{theorem}\label{Thm:H1errorestimate} Let
$u_h=\{u_0, u_b, \bu_g\}\in W_{h,k}^0$ be the approximate solution
of (\ref{1}) arising from the primal-dual weak Galerkin finite
element algorithm (\ref{2})-(\ref{32}). Assume that $a_{ij}\in
C^1(\Omega)$ and the exact solution $u$ of (\ref{1}) satisfies $u\in
H^{k+1}(\Omega)$. Then, there exists a constant $C$ such that
\begin{equation}\label{e0-H1}
\left(\sum_{T\in\T_h}\|\nabla u_0 - \nabla u\|_T^2\right)^{\frac12}
\leq Ch^{k} \|u\|_{k+1},
\end{equation}
provided that the meshsize $h$ is sufficiently small and the dual
problem (\ref{dual1})-(\ref{dual2}) has the $H^1$-regularity
estimate (\ref{regul}) with $s=0$.
\end{theorem}

\begin{proof} For any $\eta\in [C^1(\Omega)]^d$ with $\eta=0$ on
$\E_h$, let $w$ be the solution of the dual problem
(\ref{dual1})-(\ref{dual2}) with $\theta = -\nabla\cdot\eta$. Thus,
from Lemma (\ref{Lemma:TechnicalEquality}) with $v=e_h$ given as in
(\ref{error}) we obtain
\begin{equation*}
\begin{split}
-(e_0, \nabla\cdot\eta) =&\sum_{T\in{\cal T}_h} \sum_{i,j=1}^d
(a_{ij}\partial^2_{ij,d} e_h,
w)_T-\langle(e_{gi} -\partial_i e_0) n_j, ({\cal Q}_h-I)(a_{ij}w)\rangle_{\partial T}\\
&+ \langle (e_b-e_0) n_i,\partial_j( {\cal
Q}_h-I)(a_{ij}w)\rangle_{\partial T}\\
= & I_1 - I_2 + I_3,
\end{split}
\end{equation*}
where $I_j$ are defined in the obvious way. Since $\eta$ vanishes on
the wired basket $\E_h$, then from the integration by parts we have
\begin{equation}\label{2.14:800:10}
(\nabla e_0, \eta) = I_1 - I_2 + I_3.
\end{equation}
Using the two estimates in Lemma \ref{Lemma:TechnicalEstimates:01},
we can bound the terms $I_2$ and $I_3$ as follows
\begin{equation}\label{2.14:800:15}
|I_2| + |I_3| \leq C h \|\theta\|_{-1} \|e_h\|_{2,h} \leq C h
\|\eta\|_0\|e_h\|_{2,h}.
\end{equation}

As to the term $I_1$, we use the error equation (\ref{sehv2}) to
obtain
\begin{equation}\label{2.14.120}
\begin{split}
I_1 = & \sum_{T\in{\cal T}_h} \sum_{i,j=1}^d
(a_{ij}\partial^2_{ij,d} e_h, w)_T\\
=& \sum_{T\in{\cal T}_h} \sum_{i,j=1}^d (a_{ij}\partial^2_{ij,d}
e_h, {\cal Q}_h w)_T + (a_{ij}\partial^2_{ij,d} e_h, (I-{\cal Q}_h)
w)_T\\
=& \sum_{T\in{\cal T}_h} \sum_{i,j=1}^d ((I-{\cal
Q}_h)\partial_{ij}^2u, a_{ij} {\cal Q}_h w)_T+ \sum_{T\in{\cal T}_h}
\sum_{i,j=1}^d(a_{ij}\partial^2_{ij,d} e_h, (I-{\cal Q}_h) w)_T.
\end{split}
\end{equation}
Note that
\begin{equation}\label{EQ:New:2015:800}
\begin{split}|((I-{\cal Q}_h)\partial_{ij}^2u, a_{ij} {\cal Q}_h w)_T| =&
|((I-{\cal Q}_h)\partial_{ij}^2u, (I-{\cal Q}_h) a_{ij} {\cal
Q}_hw)_T|\\
\leq & \|(I-{\cal Q}_h)\partial_{ij}^2u\|_T \|(I-{\cal Q}_h) a_{ij}
{\cal Q}_hw\|_T\\
\leq & C h_T \|(I-{\cal Q}_h)\partial_{ij}^2u\|_T \|w\|_{1,T}
\end{split}
\end{equation}
and by (\ref{qaij})
\begin{equation}\label{EQ:New:2015:810}
\begin{split}
|(a_{ij}\partial^2_{ij,d} e_h, &(I-{\cal Q}_h) w)_T|
=|((a_{ij}-\bar{a}_{ij})\partial^2_{ij,d} e_h, (I-{\cal Q}_h) w)_T|\\
\leq & \ \|a_{ij}-\bar{a}_{ij}\|_{L^\infty(T)} \|\partial^2_{ij,d}
e_h\|_T \|(I-{\cal Q}_h) w\|_T\\
\leq & \ \varepsilon(h_T) h_T \|w\|_{1,T}\left( \|\partial_{ij}^2
e_0\|_T^2+s_T(e_h,e_h)\right)^{\frac12},
\end{split}
\end{equation}
where $\varepsilon(h_T) \to 0$ as $h\to 0$. Using
(\ref{EQ:New:2015:800}) and (\ref{EQ:New:2015:810}), we obtain the
following estimate for the term $I_1$:
\begin{equation}\label{EQ:New:2015:820}
\begin{split}
|I_1| \leq & C h \left(\varepsilon(h) \|\nabla^2 e_0\|_0 +
\|e_h\|_{2,h}
+ \sum_{i,j=1}^d \|(I-{\cal Q}_h)\partial^2_{ij} u\|_0 \right)\|w\|_1\\
\leq & C \left(\varepsilon(h) \|\nabla e_0\|_0 + h \|e_h\|_{2,h} + h
\sum_{i,j=1}^d \|(I-{\cal Q}_h)\partial_{ij}^2 u\|_0
\right)\|\eta\|_0,
\end{split}
\end{equation}
where we have used the inverse inequality and the estimate
$\|w\|_1\leq C \|\theta\|_{-1}\leq C \|\eta\|_0$. Substituting
(\ref{EQ:New:2015:820}) and (\ref{2.14:800:15}) into
(\ref{2.14:800:10}) yields
$$
|(\nabla e_0, \eta)| \leq C \left(\varepsilon(h) \|\nabla e_0\|_0 +
h \|e_h\|_{2,h} + h \sum_{i,j=1}^d \|(I-{\cal Q}_h)\partial_{ij}^2
u\|_0 \right)\|\eta\|_0.
$$
Since the set of all such $\eta$ is dense in $L^2(\Omega)$, then the
above inequality implies
$$
\|\nabla e_0\|_0\leq C \left(\varepsilon(h) \|\nabla e_0\|_0 + h
\|e_h\|_{2,h} + h \sum_{i,j=1}^d \|(I-{\cal Q}_h)\partial_{ij}^2
u\|_0 \right),
$$
which leads to
\begin{equation}\label{EQ:New:2015:820:100}
\|\nabla e_0\|_0\leq C h \left(\|e_h\|_{2,h} + \sum_{i,j=1}^d
\|(I-{\cal Q}_h)\partial_{ij}^2 u\|_0 \right)
\end{equation}
for sufficiently small meshsize $h$. The inequality
(\ref{EQ:New:2015:820:100}), together with the error estimate
(\ref{erres}) and the usual triangle inequality, verifies the
estimate (\ref{e0-H1}).
\end{proof}

The following is an error estimate for the primal variable $u_h$ in
the usual $L^2$ norm.

\begin{theorem}\label{Thm:L2errorestimate} Assume that each entry of the
coefficient matrix $\{a_{ij}\}_{d\times d}$ is in $C^1(\Omega)\cap
\left[ \Pi_{T\in\T_h} W^{2,\infty}(T)\right]$. In addition, assume
that the dual problem (\ref{dual1})-(\ref{dual2}) has
$H^2$-regularity with the a priori estimate (\ref{regul}) (i.e.,
$s=1$), and $P_1(T)\subset S_k(T)$ for all $T\in\T_h$. Then, there
exists a constant $C$ such that
\begin{equation}\label{e0}
\|u_0 - u\|_0 \leq Ch^{k+1} \|u\|_{k+1},
\end{equation}
provided that the meshsize $h$ is sufficiently small.
\end{theorem}

\begin{proof} The proof of this theorem is based on the result of
Theorem \ref{Thm:H1errorestimate}, and the two proofs have a lot in
common. Let $w$ be the solution of the dual problem
(\ref{dual1})-(\ref{dual2}) with $\theta\in L^2(\Omega)$. From Lemma
\ref{Lemma:TechnicalEquality} with $v=e_h$ given by (\ref{error}),
we have
\begin{equation}\label{2.14:800:10:L2}
\begin{split}
(e_0, \theta) =&\sum_{T\in{\cal T}_h} \sum_{i,j=1}^d
(a_{ij}\partial^2_{ij,d} e_h,
w)_T-\langle(e_{gi} -\partial_i e_0) n_j, ({\cal Q}_h-I)(a_{ij}w)\rangle_{\partial T}\\
&+ \langle (e_b-e_0) n_i,\partial_j( {\cal
Q}_h-I)(a_{ij}w)\rangle_{\partial T}\\
= & J_1 - J_2 + J_3,
\end{split}
\end{equation}
where $J_m$ are defined accordingly. Using the two estimates in
Lemma \ref{Lemma:TechnicalEstimates:02} we obtain the following
estimates
\begin{equation}\label{2.14:800:15:L2}
|J_2| + |J_3| \leq C h^2 \|\theta\|_{0} \|e_h\|_{2,h}.
\end{equation}

For the term $J_1$, we use the error equation (\ref{sehv2}) to
obtain
\begin{equation}\label{2.14.120:L2}
\begin{split}
J_1 = & \sum_{T\in{\cal T}_h} \sum_{i,j=1}^d
(a_{ij}\partial^2_{ij,d} e_h, w)_T\\
=& \sum_{T\in{\cal T}_h} \sum_{i,j=1}^d (a_{ij}\partial^2_{ij,d}
e_h, {\cal Q}_h w)_T + (a_{ij}\partial^2_{ij,d} e_h, (I-{\cal Q}_h)
w)_T\\
=& \sum_{T\in{\cal T}_h} \sum_{i,j=1}^d ((I-{\cal
Q}_h)\partial_{ij}^2u, a_{ij} {\cal Q}_h w)_T+ \sum_{T\in{\cal T}_h}
\sum_{i,j=1}^d(a_{ij}\partial^2_{ij,d} e_h, (I-{\cal Q}_h) w)_T.
\end{split}
\end{equation}
Since $P_1(T)\subseteq S_k(T)$ and ${\cal Q}_h$ is the $L^2$
projection onto $S_k(T)$, then
\begin{equation}\label{EQ:New:2015:800:L2}
\begin{split}|((I-{\cal Q}_h)\partial_{ij}^2u, a_{ij} {\cal Q}_h w)_T| =&
|((I-{\cal Q}_h)\partial_{ij}^2u, (I-{\cal Q}_h) a_{ij} {\cal
Q}_hw)_T|\\
\leq & \|(I-{\cal Q}_h)\partial_{ij}^2u\|_T \|(I-{\cal Q}_h) a_{ij}
{\cal Q}_hw\|_T\\
\leq & C h_T^2 \|(I-{\cal Q}_h)\partial_{ij}^2u\|_T \|w\|_{2,T}
\end{split}
\end{equation}
and by (\ref{qaij}) we arrive at
\begin{equation}\label{EQ:New:2015:810:L2}
\begin{split}
&\ |(a_{ij}\partial^2_{ij,d} e_h, (I-{\cal Q}_h) w)_T|\\
= &\ |((a_{ij}-\bar{a}_{ij})\partial^2_{ij,d} e_h, (I-{\cal Q}_h) w)_T|\\
\leq & \ \|a_{ij}-\bar{a}_{ij}\|_{L^\infty(T)} \|\partial^2_{ij,d}
e_h\|_T \|(I-{\cal Q}_h) w\|_T\\
\leq & \ C h_T^3 \|w\|_{2,T}\left( \|\partial_{ij}^2
e_0\|_T^2+s_T(e_h,e_h)\right)^{\frac12}.
\end{split}
\end{equation}
It follows from (\ref{EQ:New:2015:800:L2}) and
(\ref{EQ:New:2015:810:L2}) that
\begin{equation}\label{EQ:New:2015:820:L2}
\begin{split}
|J_1| \leq & C \left( h^3\|\nabla^2 e_0\|_0 + h^3 \|e_h\|_{2,h} +
h^2 \sum_{i,j=1}^d \|(I-{\cal Q}_h)\partial_{ij}^2 u\|_0
\right)\|w\|_2\\
\leq & C \left( h^2\|\nabla e_0\|_0 + h^3 \|e_h\|_{2,h} + h^2
\sum_{i,j=1}^d \|(I-{\cal Q}_h)\partial_{ij}^2 u\|_0
\right)\|\theta\|_0,
\end{split}
\end{equation}
where we have used the inverse inequality and the regularity
assumption (\ref{regul}) with $s=1$. Substituting
(\ref{EQ:New:2015:820:L2}) and (\ref{2.14:800:15:L2}) into
(\ref{2.14:800:10:L2}) yields
$$
|(e_0, \theta)| \leq C h^2 \left( \|\nabla e_0\|_0 + \|e_h\|_{2,h} +
\sum_{i,j=1}^d \|(I-{\cal Q}_h)\partial_{ij}^2 u\|_0
\right)\|\theta\|_0.
$$
Thus, we have
$$
\|e_0\|_0\leq C h^2 \left( \|\nabla e_0\|_0 + \|e_h\|_{2,h} +
\sum_{i,j=1}^d \|(I-{\cal Q}_h)\partial_{ij}^2 u\|_0 \right),
$$
which, together with the error estimates (\ref{erres}),
(\ref{e0-H1}), and the usual triangle inequality, gives rise to the
$L^2$-error estimate (\ref{e0}) when the meshsize $h$ is
sufficiently small. This completes the proof of the theorem.
\end{proof}

\begin{remark} The optimal order error estimate (\ref{e0}) is based on the assumption that
$P_1(T)\subseteq S_2(T)$. This assumption was used in the derivation
of the inequalities (\ref{2.14:800:15:L2}),
(\ref{EQ:New:2015:800:L2}), and (\ref{EQ:New:2015:810:L2}). In the
case of $P_1(T)\nsubseteq S_2(T)$, those inequalities need to be
modified by replacing $\|w\|_{2,T}$ by $h_T^{-1}\|w\|_{1,T}$. As a
result, the following sub-optimal order error estimate holds true
\begin{equation}\label{e0-L2-suboptimal}
\|u_0 - u\|_0 \leq Ch^{k} \|u\|_{k+1}
\end{equation}
provided that (1) the coefficient matrix $\{a_{ij}\}_{d\times d}$
satisfies $a_{ij}\in C^1(\Omega)$, (2) the meshsize $h$ is
sufficiently small, and (3) the dual problem
(\ref{dual1})-(\ref{dual2}) has the $H^1$-regularity with $s=0$ in
the a priori estimate (\ref{regul}).
\end{remark}

To establish some error estimates for the two boundary components
$u_b$ and $\bu_g$, we introduce the following norms
\begin{equation}\label{EQ:eb-eg-L2norm}
\|e_b\|_{L^2}:=\Big(\sum_{T\in {\cal T}_h} h_T\|e_b\|_{\partial
T}^2\Big)^{\frac{1}{2}},\quad \|\be_g\|_{L^2}:=\Big(\sum_{T\in {\cal
T}_h} h_T\|\be_g\|_{\partial T}^2\Big)^{\frac{1}{2}}.
\end{equation}

\begin{theorem}\label{Thm:L2errorestimate-ub}
Under the assumptions of Theorem \ref{Thm:L2errorestimate}, there
exists a constant $C$ such that
\begin{eqnarray}\label{eb}
\|u_b- Q_b u\|_{L^2} &\leq & Ch^{k+1} \|u\|_{k+1},\\
\|\bu_g - {\bf Q}_b \nabla u\|_{L^2} &\leq & Ch^{k}
\|u\|_{k+1}.\label{eg}
\end{eqnarray}
\end{theorem}

\begin{proof} On each element $T\in\T_h$, we have from the triangle
inequality that
$$
\|e_b\|_{\pT} \leq \|e_0\|_{\pT} + \|e_b-e_0\|_\pT.
$$
Thus, by the trace inequality (\ref{x}) we obtain
\begin{equation*}
\begin{split}
\sum_{T\in\T_h} h_T \|e_b\|_{\pT}^2 & \leq 2 \sum_{T\in\T_h}
h_T\|e_0\|_{\pT}^2 + C h^4 \sum_{T\in\T_h}
h_T^{-3}\|e_b-e_0\|_\pT^2\\
& \leq C (\|e_0\|^2_0 + h^4 \|e_h\|_{2,h}^2),
\end{split}
\end{equation*}
which, together with the error estimates (\ref{erres}) and
(\ref{e0}), gives rise to (\ref{eb}).

To derive (\ref{eg}), we apply the same approach to the error
component $\be_g = \bu_g -{\bf Q}_b \nabla u$ as follows
\begin{equation*}
\begin{split}
\sum_{T\in\T_h} h_T \|\be_g\|_{\pT}^2 & \leq 2 \sum_{T\in\T_h}
h_T\|\nabla e_0\|_{\pT}^2 + C h^2 \sum_{T\in\T_h}
h_T^{-1}\|\be_g-\nabla e_0\|_\pT^2\\
& \leq C (\sum_{T\in\T_h} \|\nabla e_0\|_T^2 + h^2 \|e_h\|_{2,h}^2).
\end{split}
\end{equation*}
It then follows from the error estimates (\ref{erres}) and
(\ref{e0-H1}) that (\ref{eg}) holds true.
\end{proof}

\section{Numerical Results}\label{Section:NE}
In this section, we present some numerical results for the
primal-dual WG finite element method proposed and analyzed in the
previous sections. The test problems are defined in 2D polygonal
domains in the following form: Find $u\in H^2(\Omega)$ such that
\begin{equation}\label{Test-Problem}
\begin{split}
\sum_{i,j=1}^2 a_{ij}\partial^2_{ij}u= &f,\quad \text{in}\
\Omega,\\
u = & g,\quad \text{on}\ \partial\Omega.
\end{split}
\end{equation}
For simplicity, in the numerical scheme (\ref{2})-(\ref{32}), we
shall make use of the lowest order WG element on triangular
partitions; i.e., $k=2$ in $W_{k}(T)$ on triangles $T\in\T_h$ given
by (\ref{EQ:local-weak-fem-space}). The goal is to illustrate the
efficiency and confirm the convergence theory established in the
previous sections through numerical experiments.

For the lowest order WG element with $k=2$, the corresponding finite
element spaces are given by
$$
W_{h,2}=\{v=\{v_0,v_b, \bv_g\}:\ v_0\in P_2(T), v_b\in P_2(e),
\bv_g\in [P_1(e)]^2, \forall T\in {\cal T}_h, e\in \E_h \},
$$
and
$$
S_{h,2}=\{\sigma: \ \sigma|_T \in S_2(T),\ \forall T\in {\cal T}_h
\}.
$$
A finite element function $v\in W_{h,2}$ is said to be of {\it
$C^0$-type} if $v_b= v_0|_\pT$ for each element $T$. For $C^0$-type
WG elements, the boundary component $v_b$ can be merged with $v_0$
in all the formulations since it coincides with the trace of $v_0$
on the element boundary. This clearly results in a linear system
that has less computational complexity than fully discontinuous type
WG elements. But the $C^0$ continuity limits the pool of
availability of polygonal elements due to the obvious constraints.

The local finite element space $S_2(T)$ is chosen such that $P_0(T)
\subseteq S_2(T) \subseteq P_1(T)$. Our numerical experiments are
conducted for the case of both $S_2(T)=P_1(T)$ and $S_2(T)=P_0(T)$
with $C^0$-type $W_{h,2}$. For convenience, the $C^0$-type WG
element with $S_2(T)=P_1(T)$ shall be called the
$P_2(T)/[P_1(\pT)]^2/P_1(T)$ element. Analogously, the $C^0$-type WG
element with $S_2(T)=P_0(T)$ is called the
$P_2(T)/[P_1(\pT)]^2/P_0(T)$ element.

It should be pointed out that all the theoretical results developed
in previous sections can be extended to $C^0$-type elements without
any difficulty. For $C^0$-type elements, the discrete weak second
order partial derivative $\partial^2_{ij, d} v$ should be computed
as a polynomial in $S_2(T)$ on each element $T$ by solving the
following equation
\begin{equation*}
\begin{split}
(\partial^2_{ij, d}v,\varphi)_T=-(\partial_i
v_0,\partial_{j}\varphi)_T+ \langle v_{gi} ,\varphi
n_j\rangle_{\partial T},\qquad \forall \varphi\in S_2(T).
\end{split}
\end{equation*}

Three domains are used in our numerical experiments: the unit square
$\Omega=(0,1)^2$, the reference domain $\Omega=(-1,1)^2$, and the
L-shaped domain with vertices $A_0=(0,0), \ A_1=(2,0), \ A_2=(1,1),
\ A_3=(1,2),$  and $A_4=(0,2)$. Given an initial coarse
triangulation of the domain, a sequence of triangular partitions are
obtained successively through a uniform refinement procedure that
divides each coarse level triangle into four congruent sub-triangles
by connecting the three mid-points on the edges of each triangle.

We use $u_h=\{u_0, \bu_g\}\in W_{h,2}$ and $\lambda_h\in S_{h,2}$ to
denote the primal-dual WG-FEM solution arising from
(\ref{2})-(\ref{32}). These numerical solutions are compared with
some interpolants of the exact solution in various norms.
Specifically, the numerical component $u_0$ is compared with the
standard Lagrange interpolation of the exact solution $u$ on each
triangular element by using three vertices and three mid-points on
the edge, which is denoted as $I_h u$. The vector component $\bu_g$
is compared with the linear interpolant of $\nabla u$, denoted as
${\bf I}_g (\nabla u)$, on each edge $e\in \E_h$. The Lagrange
multiplier $\lambda_h$ is compared with $\lambda=0$, as it is the
trivial solution of the dual problem. Denote their differences by
$$
e_h=\{e_0,\textbf{e}_g\}:=\{u_0- I_h u, \ \bu_g - {\bf I}_g (\nabla
u)\},\quad \gamma_h=\lambda_h-0.
$$
The following norms are used to measure the magnitude of the error:
\begin{eqnarray*}
\mbox{$L^2$- norm:}\quad & &  \|e_0\|_0=\Big(\sum_{T\in {\cal T}_h}
\int_T |e_0|^2 dT\Big)^{\frac{1}{2}},\\
\mbox{$H^1$-seminorm:}\quad  & &
\|\textbf{e}_g\|_{L^2}=\Big(\sum_{T\in {\cal T}_h} h_T
\int_{\partial T}
|\textbf{e}_g|^2 ds\Big)^{\frac{1}{2}},\\
\mbox{$L^2$-norm:}\quad & &  \|\gamma_h\|_0=\Big(\sum_{T\in {\cal
T}_h} \int_T |\gamma_h|^2 dT\Big)^{\frac{1}{2}}.
\end{eqnarray*}

\subsection{Numerical experiments with continuous coefficients}

Tables \ref{NE:TRI:Case2-1}--\ref{NE:TRI:Case2-2} illustrate the
performance of the primal-dual WG finite element method for the test
problem (\ref{Test-Problem}) with exact solution given by
$u=\sin(x_1)\sin(x_2)$ on the unit square domain and the L-shaped
domain. The right-hand side function and the Dirichlet boundary
condition are chosen to match the exact solution. The results
indicate that the convergence rates for the solution of the weak
Galerkin algorithm (\ref{2})-(\ref{32}) is of order $r=4.0$ and
$r=3.5$ in the discrete $L^2$-norm for $u_0$ on the unit square
domain and the L-shaped domain, respectively. For the discrete
$H^1$-seminorm (i.e., the $L^2$ norm for $\be_g$), the numerical
order of convergence is $r=2.0$ on both domains. For the Lagrange
multiplier $\lambda_h$, the numerical order of convergence is
$r=1.0$ in the $L^2$-norm on the square and the L-shaped domain. In
comparison, the theoretical order of convergence for $u_0$ in the
$L^2$-norm is $r=3.0$, and that for $\bu_g$ and $\lambda_h$ are
$r=2.0$ and $r=1.0$, respectively for the unit square domain. For
the L-shaped domain, the theoretical rate of convergence for $u_0$
in the $L^2$-norm should be between $r=2$ and $r=3$ due to the lack
of needed $H^2$-regularity for the dual problem
(\ref{dual1})-(\ref{dual2}). However, the theoretical rates of
convergence for $\bu_g$ and $\lambda_h$ remain to be of order
$r=2.0$ and $r=1.0$, respectively. It is clear that the numerical
results are in good consistency with the theory for $\bu_g$ and
$\lambda_h$, but greatly outperform the theory for $u_0$ in the
discrete $L^2$-norm. We believe that the primal-dual weak Galerkin
finite element method has a superconvergence for smooth solutions
with smooth data on uniform triangular partitions.

\begin{table}[h!]
\begin{center}
\caption{Convergence rates for the $C^0$-
$P_2(T)/[P_1(\pT)]^2/P_1(T)$ element applied to problem
(\ref{Test-Problem}) with exact solution $u=\sin(x_1)\sin(x_2)$ on
$\Omega=(0,1)^2$. The coefficient matrix is $a_{11}=3$,
$a_{12}=a_{21}=1$, and $a_{22}=2$.}\label{NE:TRI:Case2-1}
\begin{tabular}{|c|c|c|c|c|c|c|}
\hline
$1/h$  & $\|e_0\|_0 $ & order &  $\|\be_g\|_{L^2}$  & order  &   $\|\gamma_h\|_0$  & order  \\
\hline
1   &   0.00624 &    & 0.126    & & 0.0335   &\\
\hline
2 & 0.00147&    2.09    &0.0448&    1.50    &0.0650&    -0.96 \\
\hline
4 &  1.39e-004   &3.40   &0.0116&    1.95    &0.0284     &1.20\\
\hline
8 &  1.03e-005   &3.75   &0.00284 &  2.03    &   0.0132  &1.10 \\
\hline
16 & 6.95e-007   &3.89   &7.02e-004  &2.02&  0.00643 &1.04  \\
\hline
32 & 4.52e-008   &3.94   &1.75e-004  &2.01&  0.00317     &1.02\\
\hline
\end{tabular}
\end{center}
\end{table}

\begin{table}[h!]
\begin{center}
\caption{Convergence rates for the $C^0$-
$P_2(T)/[P_1(\pT)]^2/P_1(T)$ element applied to problem
(\ref{Test-Problem}) with exact solution $u=\sin(x_1)\sin(x_2)$ on
the L-shaped domain. The coefficient matrix is $a_{11}=3$,
$a_{12}=a_{21}=1$, and $a_{22}=2$.}\label{NE:TRI:Case2-2}
\begin{tabular}{|c|c|c|c|c|c|c|}
\hline
$1/h$        & $\|e_0\|_0 $ & order &  $\|\be_g\|_{L^2} $  & order  &   $\|\gamma_h\|_0$  & order  \\
\hline
1    & 0.0168   & & 0.481   & &0.448     & \\
\hline
2  &0.00248    &2.76   &0.125& 1.95 &0.195     &1.20\\
\hline
4 &  2.30e-004&  3.43    &0.0310     &2.01   &0.0875     &1.16 \\
\hline
8 &  1.93e-005&  3.57    &0.00767    &2.01 & 0.0413  &1.08 \\
\hline
16 & 1.61e-006&  3.59 &  0.00191 &   2.01    &0.0202 &   1.03 \\
\hline 3.2 &1.37e-007&3.56 &4.75e-004&2.00&0.00999 &1.01
\\
\hline
\end{tabular}
\end{center}
\end{table}

Table \ref{NE:TRI:12-1} contains some numerical results for the
problem (\ref{Test-Problem}) in $\Omega=(-1,1)^2$ with exact
solution $u=\sin(x_1)\sin(x_2)$ with varying coefficients. Observe
that the coefficient function
$a_{12}=0.5|x_1|^{\frac13}|x_2|^{\frac13}$ is continuous in the
domain, but its derivative has a singularity at the origin so that
the corresponding second order elliptic equation can not be written
in a divergence form. The performance of the primal-dual WG finite
element method is similar to the case of constant coefficient
matrix, except that the superconvergence seems to be weakened in the
convergence order.

\begin{table}[h!]
\begin{center}
\caption{Convergence rates for the $C^0$-
$P_2(T)/[P_1(\pT)]^2/P_1(T)$ element applied to problem
(\ref{Test-Problem}) with exact solution $u=\sin(x_1)\sin(x_2)$ on
the domain $(-1,1)^2$. The coefficient matrix is $a11=1+|x_1|$,
$a12=a21=0.5|x_1|^{\frac13}|x_2|^{\frac13}$,
$a22=1+|x_2|$.}\label{NE:TRI:12-1}
\begin{tabular}{|c|c|c|c|c|c|c|}
\hline
$2/h$        & $\|e_0\|_0 $ & order &  $\|\be_g\|_{L^2} $  & order  &   $\|\gamma_h\|_0$  & order  \\
\hline
1 & 0.1763728 &&    1.2455105  &&   0.0038959 &\\
\hline 2 & 0.0356693 & 2.31 & 0.4859078 & 1.36 &
0.0082045 & -1.07 \\
\hline 4 & 0.0036026 & 3.31 & 0.1304043 &
1.90 & 0.0032424 & 1.34 \\
\hline 8 & 2.78e-004 &  3.70 & 0.0318454 & 2.03 &
0.0015142 & 1.10 \\
\hline 16 & 2.02e-005 & 3.78 & 0.0078262 & 2.02 & 7.42e-004
& 1.03 \\
\hline 32 & 2.37e-006 & 3.09 & 0.00194 & 2.01 & 3.68e-004 & 1.01 \\
\hline
\end{tabular}
\end{center}
\end{table}

\medskip

In Table \ref{NE:TRI:12-2}, we present some numerical results for
the test problem (\ref{Test-Problem}) with exact solution
$u=\sin(x_1)\sin(x_2)$ in $\Omega=(-1,1)^2$ when the
$C^0$-$P_2(T)/[P_1(\pT)]^2/P_0(T)$ element is employed in the
primal-dual WG finite element scheme (\ref{2})-(\ref{32}). Note that
the Lagrange multiplier $\lambda$ is now approximated by piecewise
constant functions; i.e., $S_2(T)=P_0(T)$. The results indicate that
the numerical solution $\bu_g$ converges to the exact solution
$\nabla u$ at the rate of $r=2.0$ in the usual $L^2$ norm. The same
rate of convergence is also observed for $u_h-u$ in the $L^2$-norm.
The Lagrange multiplier has a convergence rate slightly higher than
$r=1.0$ to the exact solution of $\lambda=0$. The numerical
convergence for the primal variable $u$ is in great consistency with
the theory developed in this paper, while the convergence for the
dual variable $\lambda$ outperforms the theory of $r=1.0$.

\begin{table}[h!]
\begin{center}
\caption{Convergence rates for the $C^0$-
$P_2(T)/[P_1(\pT)]^2/P_0(T)$ element applied to problem
(\ref{Test-Problem}) with exact solution $u=\sin(x_1)\sin(x_2)$ on
the domain $(-1,1)^2$. The coefficient matrix is $a11=1+|x_1|$,
$a12=a21=0.5|x_1|^{\frac13}|x_2|^{\frac13}$,
$a22=1+|x_2|$.}\label{NE:TRI:12-2}
\begin{tabular}{|c|c|c|c|c|c|c|}
\hline
$2/h$        & $\|e_0\|_0 $ & order &  $\|\be_g\|_{L^2} $  & order  &   $\|\gamma_h\|_0$  & order  \\
\hline

1 & 2.80e-006  &&     1.7557720  &&   2.10e-006 & \\
\hline

2 & 0.1756863 & -15.94 & 0.6755226 & 1.38 & 0.0894908 & -15.38\\
\hline

4 & 0.0395431 & 2.15 & 0.1637125 & 2.04  &  0.0517686 & 0.79
\\ \hline

8 &  0.0089637 & 2.14 & 0.0386493 & 2.08 & 0.0190018 & 1.45 \\
\hline

16 & 0.0021665 & 2.05 & 0.0093809 & 2.04 & 0.0068545 & 1.47
\\
\hline 32 & 5.37e-004 & 2.01 & 0.00231 & 2.02 & 0.00288 & 1.25 \\
\hline
\end{tabular}
\end{center}
\end{table}

\subsection{Numerical experiments with discontinuous coefficients}
In the second part of the numerical experiment, we consider problems
with discontinuous coefficients that satisfy the Cord\`es condition
(\ref{cordes}). The first such problem is given as follows
\begin{equation}\label{EQ:NE:500}
\begin{split}
\sum_{i,j=1}^2 (1+\delta_{ij}) \frac{x_i}{|x_i|}\frac{x_j}{|x_j|}
\partial^2_{ij} u & = f\qquad \mbox{in } \Omega,\\
u & = 0\qquad \mbox{on } \partial\Omega,
\end{split}
\end{equation}
where $\Omega=(-1,1)^2$ is the reference square domain and the
function $f$ is chosen so that the exact solution of
(\ref{EQ:NE:500}) is
\begin{equation}\label{EQ:NE:501}
u= x_1 x_2 \left(1-e^{1-|x_1|}\right)\left(1-e^{1-|x_2|}\right).
\end{equation}
It is not hard to see that the Cord\`es condition (\ref{cordes}) is
satisfied for the problem (\ref{EQ:NE:500}) with $\varepsilon = 3/5$
and the coefficients matrix is discontinuous across the $x_1$- and
$x_2$-axis. This is a test problem suggested in \cite{smears}.

Table \ref{NE:TRI:Case10-1} contains some numerical results for the
test problem (\ref{EQ:NE:500}) when the
$C^0$-$P_2(T)/[P_1(\pT)]^2/P_1(T)$ element is employed in the WG
finite element scheme (\ref{2})-(\ref{32}). Note that the Lagrange
multiplier $\lambda$ is approximated by piecewise linear functions;
i.e., $S_2(T)=P_1(T)$. The results indicate that the numerical
solution $\bu_g$ converges to the exact solution $\nabla u$ at the
rate of $r=2.0$ in the usual $L^2$ norm, which is consistent with
the theoretical rate of convergence. The Lagrange multiplier has a
convergence rate that seems to be higher than the theory-predicted
rate of $r=1.0$. For the approximation of $u$, the convergence rate
in the usual $L^2$ norm seems to exceed $r=2$. It should be pointed
out that there is no theoretical result on optimal order of error
estimates for $u-u_h$ in the $L^2$ norm, as it is not clear if the
dual problem (\ref{dual1})-(\ref{dual2}) has the required regularity
necessary for carrying out the convergence analysis. Table
\ref{NE:TRI:Case10-1} shows that the numerical performance of the
primal-dual WG finite element method is typically better than what
theory predicts.

\begin{table}[h!]
\begin{center}
\caption{Convergence rates for the $C^0$-
$P_2(T)/[P_1(\pT)]^2/P_1(T)$ element applied to problem
(\ref{EQ:NE:500}) with exact solution given by
(\ref{EQ:NE:501}).}\label{NE:TRI:Case10-1}
\begin{tabular}{|c|c|c|c|c|c|c|}
\hline
$2/h$        & $\|e_0\|_0 $ & order &  $\|\be_g \|_{L^2} $  & order  &   $\|\gamma_h\|_0$  & order  \\
\hline 1& 0.094005  & & 0.765566 &  & 0.337760
 &\\
\hline
2 & 0.248887 & -1.40& 1.346963 & -0.82  &0.642055 &-0.93\\
\hline
4 & 0.106414 & 1.23 & 0.538155 & 1.32 & 1.284597 &-1.0 \\
\hline
8 &  0.030602  & 1.80 &  0.137486 &1.97 & 0.537170&1.26 \\
\hline
16 & 0.007488 &2.03 & 0.032750  &2.07 & 0.212136&1.34\\
\hline
32 & 0.001736 &2.11 & 0.007848  &2.06 & 0.092301&1.20\\
\hline
\end{tabular}
\end{center}
\end{table}

In Table \ref{NE:TRI:Case10-2}, we present some numerical results
for the test problem (\ref{EQ:NE:500}) when the
$C^0$-$P_2(T)/[P_1(\pT)]^2/P_0(T)$ element is employed in the WG
finite element scheme (\ref{2})-(\ref{32}). It is interesting to
note that the absolute error for each numerical approximation is
smaller than those arising from the use of
$C^0$-$P_2(T)/[P_1(\pT)]^2/P_1(T)$ element in Table
\ref{NE:TRI:Case10-1}, while the rate of convergence remains to be
comparable. Readers are invited to draw their own conclusions for
the results illustrated in this table.

\begin{table}[h!]
\begin{center}
\caption{Convergence rates for the $C^0$-
$P_2(T)/[P_1(\pT)]^2/P_0(T)$ element applied to problem
(\ref{EQ:NE:500}) with exact solution given by
(\ref{EQ:NE:501}).}\label{NE:TRI:Case10-2}
\begin{tabular}{|c|c|c|c|c|c|c|}
\hline
$2/h$        & $\|e_0\|_0 $ & order &  $\|\be_g \|_{L^2} $  & order  &   $\|\gamma_h\|_0$  & order  \\
\hline
1&0.0393&&0.672&&0.137&\\
\hline
2&0.0322&0.28&0.322&1.06&0.104&0.40\\
\hline
4&0.00750&2.10&0.0791&2.03&0.0532&0.96\\
\hline
8&0.00161&2.22&0.0180&2.13&0.0204&1.39\\
\hline
16&3.85e-004&2.07&0.00427&2.08&0.00818&1.32\\
\hline
32&9.52e-005&2.02&0.00104&2.04&0.00371&1.14\\
\hline
\end{tabular}
\end{center}
\end{table}

The final test equation is given by
\begin{equation}\label{EQ:NE:800}
\sum_{i,j=1}^2\left(\delta_{ij}+\frac{x_ix_j}{|x|^2}\right)
\partial_{ij}^2 u=f\qquad \mbox{in } \Omega,
\end{equation}
where $|x| = \sqrt{x_1^2+x_2^2}$ is the length of $x$. Note that the
coefficient $a_{ij}=\frac{x_ix_j}{|x|^2}$ fails to be continuous at
the origin for $i\neq j$. For $\alpha>1$, it can be seen that
$u=|x|^\alpha\in H^2(\Omega)$ satisfies (\ref{EQ:NE:800}) with
$f=(2\alpha^2-\alpha)|x|^{\alpha-2}$.  The linear operator in
(\ref{EQ:NE:800}) satisfies the Cord\`es condition with
$\varepsilon=4/5$. The solution $u=|x|^\alpha$ has the regularity of
$H^{1+\alpha-\tau}(\Omega)$ for arbitrarily small $\tau>0$. In the
numerical experiments, we take $\alpha=1.6$ with problem
(\ref{EQ:NE:800}) defined on two square domains: $(0,1)^2$ and
$(-1,1)^2$. The case of $\Omega=(0,1)^2$ was tested in
\cite{smears}.

Tables \ref{NE:TRI:test2-2} and \ref{NE:TRI:test2-6} illustrate the
performance of the primal-dual WG scheme for the domain
$\Omega=(0,1)^2$. Note that the coefficient matrix
$\{a_{ij}\}_{2\times 2}$ is continuous in the interior of the
domain, but it fails to be continuous at the corner point $A=(0,0)$.
The numerical approximation suggests a convergence rate of $r=1.6$
in the $H^1$-seminorm (i.e., $L^2$ for $\be_g$) and $r=0.6$ in $L^2$
for the Lagrange multiplier $\lambda_h$. These are in great
consistency with theory developed in earlier sections, as the
solution $u=|x|^{1.6}$ has the regularity of $H^{2.6 -
\tau}(\Omega)$ for any small $\tau
>0$. It seems that the $L^2$ norm for $u-u_h$ has a numerical
convergence rate of $r=2$, for which no theory was available to
apply or compare with.


\begin{table}[h!]
\begin{center}
\caption{Convergence rates for the $C^0$-
$P_2(T)/[P_1(\pT)]^2/P_1(T)$ element applied to problem
(\ref{EQ:NE:800}) on $\Omega=(0,1)^2$ with exact solution
$u=|x|^{1.6}$.}\label{NE:TRI:test2-2}
\begin{tabular}{|c|c|c|c|c|c|c|}
\hline
$1/h$        & $\|e_0\|_0 $ & order &  $\|\be_g \|_{L^2} $  & order  &   $\|\gamma_h\|_0$  & order  \\
 \hline
1&0.020  &&0.315&&0.304 &\\
\hline
2& 0.00629 &1.68 &0.126&1.32 &0.248&0.296\\
\hline
4& 0.00174&1.86&0.0446&1.50&0.182 &0.445\\
\hline
8& 4.43e-004&1.97 &0.0152&1.56&0.126&0.537 \\
\hline
16& 1.08e-004&2.03  &0.00508&1.58&0.0846&0.570\\
\hline
32& 2.60e-005&2.05 &0.00169&1.59&0.0564&0.584 \\
\hline
\end{tabular}
\end{center}
\end{table}

\begin{table}[h!]
\begin{center}
\caption{Convergence rates for the $C^0$-
$P_2(T)/[P_1(\pT)]^2/P_0(T)$ element applied to problem
(\ref{EQ:NE:800}) on $\Omega=(0,1)^2$ with exact solution
$u=|x|^{1.6}$.}\label{NE:TRI:test2-6}
\begin{tabular}{|c|c|c|c|c|c|c|}
\hline
$1/h$        & $\|e_0\|_0 $ & order &  $\|\be_g \|_{L^2} $  & order  &   $\|\gamma_h\|_0$  & order  \\
\hline
1& 0.00405&&0.489&&0.0623&\\
\hline
2& 0.00803&-0.988&0.177&1.46&0.0616&0.0156\\
\hline
4& 0.00263&1.61&0.0616&1.53&0.0476 &0.372\\
\hline
8& 7.90e-004&1.74&0.0210&1.55&0.0327&0.544\\
\hline
16& 2.20e-004&1.85&0.00705&1.57&0.0218&0.582\\
\hline
32& 5.85e-005&1.91&0.00235&1.59&0.0145&0.593\\
\hline
\end{tabular}
\end{center}
\end{table}

Tables \ref{NE:TRI:test3-2} and \ref{NE:TRI:test3-6} illustrate the
performance of the primal-dual WG finite element scheme
(\ref{2})-(\ref{32}) for the equation (\ref{EQ:NE:800}) in the
domain $\Omega=(-1,1)^2$. For this test problem, the coefficient
matrix $\{a_{ij}\}_{2\times 2}$ is discontinuous at the center of
the domain so that the duality argument in the convergence theory is
not applicable. Consequently, the corresponding numerical results
are less accurate than the case of $\Omega=(0,1)^2$ as shown in
Tables \ref{NE:TRI:test2-2} and \ref{NE:TRI:test2-6}. However, the
numerical approximation suggests a convergence rate of $r=0.6$ in
$L^2$ for the Lagrange multiplier $\lambda_h$ which is consistent
with the theory. The convergence in $H^1$ and $L^2$ norms seems to
have a rate of $r=1.0$ or slightly higher.

\begin{table}[h!]
\begin{center}
\caption{Convergence rates for the $C^0$-
$P_2(T)/[P_1(\pT)]^2/P_1(T)$ element applied to problem
(\ref{EQ:NE:800}) on $\Omega=(-1,1)^2$ with exact solution
$u=|x|^{1.6}$.}\label{NE:TRI:test3-2}
\begin{tabular}{|c|c|c|c|c|c|c|}
\hline
$2/h$        & $\|e_0\|_0 $ & order &  $\|\be_g \|_{L^2} $  & order  &   $\|\gamma_h\|_0$  & order  \\
 \hline
1& 0.532&&0.511&&0.280&\\
 \hline
2&0.266 &1.00&0.403&0.344&0.623&-1.15\\
 \hline
4&0.117 &1.19&0.211&0.933 &0.562&0.149\\
 \hline
8&0.0563 &1.05 &0.111&0.927 &0.405 &0.471\\
 \hline
16&0.0271&1.06&0.0576&0.945 &0.277 &0.547\\
 \hline
32&0.0129&1.07 &0.0290 &0.987 &0.187&0.572\\
\hline
\end{tabular}
\end{center}
\end{table}

\begin{table}[H]
\begin{center}
\caption{Convergence rates for the $C^0$-
$P_2(T)/[P_1(\pT)]^2/P_0(T)$ element applied to problem
(\ref{EQ:NE:800}) on $\Omega=(-1,1)^2$ with exact solution
$u=|x|^{1.6}$.}\label{NE:TRI:test3-6}
\begin{tabular}{|c|c|c|c|c|c|c|}
\hline
$2/h$        & $\|e_0\|_0 $ & order &  $\|\be_g \|_{L^2} $  & order  &   $\|\gamma_h\|_0$  & order  \\
\hline
1&  0.647  &&0.487 &&0.0862&\\
\hline
2& 0.611&0.08 &0.697&-0.517&0.0769&0.165\\
\hline
4& 0.254 &1.26 &0.407 &0.774 &0.0500 &0.619 \\
\hline
8& 0.113&1.18&0.218&0.903&0.0417&0.264 \\
\hline
16& 0.0512&1.14&0.110 &0.984 &0.0297&0.490 \\
\hline
32 & 0.0235 &1.12&0.0540& 1.03& 0.0201&0.561\\
\hline
\end{tabular}
\end{center}
\end{table}


\end{document}